\date{} % % March 17 2003
\title{Gaussian free field light cones and $\SLE_\kappa(\rho)$}
\author{Jason Miller and Scott Sheffield}
\documentclass[12pt,naturalnames]{article}
\usepackage{amsmath}
\usepackage{amssymb}
\usepackage{amsthm}
\usepackage{amsfonts}
\usepackage{graphicx}
\usepackage{tabularx}
\usepackage{subfigure}
\usepackage{enumerate}
\usepackage{color}
\usepackage{xparse}
\usepackage{comment}
\usepackage[normalem]{ulem}
\usepackage{microtype}

\usepackage{booktabs}
\newcommand{\head}[1]{\textnormal{\textbf{#1}}}

\def\@rst #1 #2other{#1}
\newcommand\MR[1]{\relax\ifhmode\unskip\spacefactor3000 \space\fi
  \MRhref{\expandafter\@rst #1 other}{#1}}
\newcommand{\MRhref}[2]{\href{http://www.ams.org/mathscinet-getitem?mr=#1}{MR#2}}

\usepackage[margin=1.2in]{geometry}

\allowdisplaybreaks

\newif\ifhyper\IfFileExists{hyperref.sty}{\hypertrue}{\hyperfalse}

\ifhyper\usepackage{hyperref}\fi

\newif\ifdraft
\drafttrue
\numberwithin{equation}{section}
\numberwithin{figure}{section}

\newtheorem{theorem}{Theorem}
\numberwithin{theorem}{section}
\newtheorem{corollary}[theorem]{Corollary}
\newtheorem{lemma}[theorem]{Lemma}
\newtheorem{proposition}[theorem]{Proposition}

\theoremstyle{remark}
\theoremstyle{remark}\newtheorem{remark}[theorem]{Remark}

\newcommand{\R}{\mathbf{R}}
\newcommand{\C}{\mathbf{C}}
\newcommand{\D}{\mathbf{D}}

\newcommand{\N}{\mathbf{N}}
\newcommand{\Q}{\mathbf{Q}}

\newcommand{\h}{\mathbf{H}}

\newcommand{\CF}{\mathcal {F}}
\newcommand{\CI}{\mathcal {I}}
\newcommand{\CJ}{\mathcal {J}}

\newcommand{\CP}{\mathcal {P}}

\newcommand{\CR}{\mathcal {R}}

\def\diam{\mathop{\mathrm{diam}}}

\def\dist{\mathop{\mathrm{dist}}}

\newcommand{\SLE}{{\rm SLE}}

\newcommand{\ol}{\overline}
\newcommand{\ul}{\underline}
\newcommand{\wh}{\widehat}
\newcommand{\wt}{\widetilde}

\newcommand{\dimH}{\mathrm{dim}_{\mathcal H}}

\newcommand{\PV}{{\rm P.V.}}

\newcommand{\BESQ}{{\mathrm {BESQ}}}
\newcommand{\BES}{{\mathrm {BES}}}

\definecolor{purple}{rgb}{0.7,0,0.7}
\definecolor{gray}{rgb}{0.6,0.6,0.6}
\definecolor{dgreen}{rgb}{0.0,0.4,0.0}
\definecolor{dblue}{rgb}{0.0,0.0,0.5}

\newcommand{\lightcone}{{\mathbf L}}
\newcommand{\openL}{x}
\newcommand{\closeL}{y}
\newcommand{\open}[1]{\openL(#1)}
\newcommand{\close}[1]{\closeL(#1)}

\NewDocumentCommand{\pocket}{m+g}{
\IfNoValueTF{#2}
 {P(#1)}
 {P_{#1}(#2)}
}

\NewDocumentCommand{\side}{m+g}{
\IfNoValueTF{#2}
 {S_{#1}}
 {S_{#1}(#2)}
}

\NewDocumentCommand{\sideflow}{m+g}{
\IfNoValueTF{#2}
 {\eta_{#1}}
 {\eta_{#1}(#2)}
}

\def\Ito/{It\^o}

\begin{document} 
\maketitle

\begin{abstract}
We derive a surprising correspondence between $\SLE_{\kappa}(\rho)$ processes and {\em light cones} associated to the Gaussian free field (GFF).

Recall that (one-sided, chordal, origin-seeded) $\SLE_\kappa(\rho)$ processes are in some sense the simplest and most natural variants of the Schramm-Loewner evolution.  They were originally defined only for $\rho > -2$, but one can use L\'evy compensation to extend the definition to any $\rho > -2-\tfrac{\kappa}{2}$ and to obtain qualitatively different curves. The triangle $\mathcal T = \{(\kappa, \rho): (-2-\tfrac{\kappa}{2})\vee (\tfrac{\kappa}{2}-4) < \rho < -2 \}$ is the primary focus of this paper. When $(\kappa, \rho) \in \mathcal T$, the $\SLE_\kappa(\rho)$ curves are highly non-simple (and double points are dense) even though $\kappa < 4$.

Let $h$ be an instance of the GFF. Fix $\kappa\in (0,4)$ and $\chi = 2/\sqrt{\kappa} - \sqrt{\kappa}/2$. Recall that an {\em imaginary geometry ray} is a flow line of $e^{i(h/\chi +\theta)}$ that looks locally like $\SLE_\kappa$. The {\em light cone} with parameter $\theta \in [0, \pi]$ is the set of points reachable from the origin by a sequence of rays with angles in $[-\theta/2, \theta/2]$.  It is known that when $\theta=0$, the light cone looks like $\SLE_\kappa$, and when $\theta = \pi$ it looks like the range of an $\SLE_{16/\kappa}$ {\em counterflow line}. We find that when $\theta \in (0, \pi)$ the light cones are either fractal carpets with a dense set of holes or space-filling regions with no holes.

We show that every non-space-filling light cone (with $\theta \in (0,\pi)$ and $\kappa \in (0,4)$) agrees in law with the range of an $\SLE_\kappa(\rho)$ process with $(\kappa, \rho) \in \mathcal T$. Conversely, the range of any $\SLE_\kappa(\rho)$ with $(\kappa,\rho) \in {\mathcal T}$ agrees in law with a non-space-filling light cone. As a consequence of our analysis, we obtain the first proof that these $\SLE_\kappa(\rho)$ processes are a.s.\ continuous curves and show that they can be constructed as natural path-valued functions of the GFF.
\end{abstract}

\newpage
\tableofcontents
\newpage

\medbreak {\noindent\bf Acknowledgements.}  JM and SS were respectively partially supported by NSF grants DMS-1204894 and DMS-1209044.  We also thank Wendelin Werner for helpful discussions.

\section{Introduction}

The $\SLE_\kappa(\rho)$ processes are an important variant of the Schramm-Loewner evolution ($\SLE$) \cite{S0}.  They were first introduced by Lawler, Schramm, and Werner in \cite[Section~8.3]{LSW_RESTRICTION}.  Like ordinary $\SLE_\kappa$, $\SLE_\kappa(\rho)$ is defined using the Loewner equation and a {\em driving function}~$W$ that looks (at least locally) like~$\sqrt{\kappa}$ times a Brownian motion. However, in addition to the driving function~$W$, one keeps track of a so-called {\em force point} process~$V$, which itself evolves according to Loewner evolution, and which exerts a {\em drift} on~$W$ proportional to $\rho/(W-V)$.  When $\rho > 0$ (resp.\ $\rho < 0$), the drift pushes~$W$ away from (resp.\ towards) the force point~$V$, and the case $\rho = 0$ corresponds to ordinary $\SLE_\kappa$. The difference $W - V$ evolves as a positive multiple of a Bessel process of dimension $\delta(\kappa,\rho) = 1+\tfrac{2(\rho+2)}{\kappa}$.  See Section~\ref{sec::preliminaries} for a formal definition of $\SLE_\kappa(\rho)$. Various flavors of $\SLE_\kappa(\rho)$ have been discussed in the literature, but in this paper we generally assume that the processes are {\em chordal} (so they grow from~$0$ to~$\infty$ in the upper half plane~$\h$), {\em one-sided} (so that all excursions of $W-V$ away from zero have the same sign) and {\em origin seeded} (meaning that $V_0 = W_0 = 0$).

The time evolution of~$W$ and~$V$ is straightforward to define during intervals of time in which $W_t \neq V_t$, but to continue the evolution after~$W$ and~$V$ collide, one has to work out precisely how these processes ``bounce off'' one another. In the original construction in \cite[Section~8.3]{LSW_RESTRICTION}, and in most of the later work on $\SLE_\kappa(\rho)$ processes, this is only done for $\rho > -2$.  The threshold~$-2$ corresponds to $\delta(\kappa,\rho)=1$, which is the critical threshold below which Bessel processes fail to be semimartingales \cite[Chapter~11]{RY04}. This is related to the fact that $\delta > 1$ is necessary in order for the integral $\int_0^T (W_t - V_t)^{-1} dt$ to be a.s.\ finite for all $T$, which in turn ensures that the cumulative amount of drift exerted on $W$ (up to any finite time) is a.s.\ finite.

To define $\SLE_{\kappa}(\rho)$ when $\rho < -2$ it is necessary to introduce a local time L\'evy compensation to keep the accumulated drift from sending $W$ off to $\infty$ in finite time. As we recall in Section~\ref{sec::preliminaries} (citing~\cite[Section 3.2]{SHE_CLE}), there is a natural scale-invariant way to do this if and only if $\rho > -2 - \tfrac{\kappa}{2}$ so that $\delta > 0$. As detailed in~\cite[Section 3]{SHE_CLE}, if one parameterizes $W$ by the local time associated to $\{t: W_t = V_t \}$ one obtains a skew stable L\'evy process, so that the classification of general $\SLE_{\kappa}(\rho)$ processes is closely related to the classification of skew stable L\'evy processes.\footnote{In the account in \cite{SHE_CLE}, there is a parameter $\beta$ such that each $W-V$ excursion away from zero is (independently of all others) assigned a positive sign with probability $(1+\beta)/2$ and a negative sign otherwise. When $\rho = -2$, it is necessary to take $\beta =0$ to obtain a canonical, scale-invariant and non-trivial process, and there is an additional free parameter $\mu$ in that case.  We will not consider the $\rho = -2$ setting here, except to say that in some limiting sense $\beta = 1$ and $\rho = -2$ corresponds to a trivial boundary tracing path.  As mentioned above,  this paper treats only the ``one-sided'' case $\beta = 1$, and our main results assume $\rho < -2$.}

The continuity and reversibility properties of $\SLE_\kappa(\rho)$ with $\rho > -2$ are established in \cite{MS_IMAG,MS_IMAG2,MS_IMAG3,MS_IMAG4}, which exhibit and make use of explicit couplings between these processes and the Gaussian free field (GFF) \cite{She_SLE_lectures,DUB_PART,SchrammShe10,MS_IMAG,MS_IMAG4} (see also \cite{Z_R_KAPPA_RHO,DUB_DUAL} for the reversibility of $\SLE_\kappa(\rho)$ for $\kappa \in (0,4)$ and $\rho \geq \tfrac{\kappa}{2}-2$).  When $\rho > -2$, the range of an $\SLE_\kappa(\rho)$ process looks locally like the range of an ordinary $\SLE_\kappa$, except where the path hits the boundary.

When $\rho \leq -2$, however, one obtains interesting and qualitatively different processes. The Bessel dimension interval $\delta \in (0,1)$ corresponds to $\rho \in (-2 - \tfrac{\kappa}{2}, -2)$.  In this article we focus on the set $\mathcal T = \{(\kappa, \rho): (-2-\tfrac{\kappa}{2})\vee (\tfrac{\kappa}{2}-4) < \rho < -2 \}$, which corresponds to the yellow {\em light cone} region depicted in Figure~\ref{fig::rho_kappa_chart}. The {\em loops on trunk} regions shown in Figure~\ref{fig::rho_kappa_chart}. are studied in detail in \cite{cle_percolations}.\footnote{In the loops-on-trunk regime explored in \cite{cle_percolations}, each excursion of $W-V$ away from zero describes a loop, and it is important and relevant to consider {\em non-one-sided} $\SLE_\kappa(\rho)$, which can be written $\SLE_\kappa^\beta(\rho)$ for $\beta \in [-1,1]$, and which correspond to different types of CLE explorations. These explorations are useful for understanding CLE percolation and the continuum FK correspondence, among other things. In general,  $\SLE_\kappa^\beta(\rho)$ can be defined for all $\beta \in [-1,1]$ whenever $\rho \in (-2-\kappa/2, \kappa/2-2) \setminus \{-2\}$, so that $\delta \in (0,2) \setminus \{1 \}$, and \cite[Section~10.1.3]{cle_percolations} briefly describes how to interpret and prove continuity results for these processes for general $\beta$ in the case $\kappa>4$.
When $\kappa \leq 4$, it remains an open problem to prove continuity for $\SLE_{\kappa}^\beta(\rho)$ when $\beta \in (-1,1)$ and $\rho \in (-2-\kappa/2 \vee \kappa/2-4, \kappa/2-2) \setminus \{-2 \}$, i.e., in the light cone region and (the boundary-intersecting part of) the ordinary flow line region in Figure~\ref{fig::rho_kappa_chart}. We remark that in these regions, each excursion of $W-V$ away from zero should (assuming continuity of the overall path) describe a {\em chord} (i.e., a simple path segment starting and ending at different points) and we are not aware of a natural interpretation of an overall path that alternates between left and right going chords. As mentioned earlier, we treat only the case $\beta = 1$ in this paper. (The case $\beta = -1$ is equivalent by symmetry.)
}
We will find that $\SLE_\kappa(\rho)$ with $(\kappa, \rho) \in \mathcal T$ can be naturally coupled with an instance of the GFF, and that in this coupling the field a.s.\ determines the path.  This will be accomplished by showing that such a process can be realized as an {\bf ordered light cone} of angle-varying flow lines of the (formal) vector-field $e^{i h / \chi}$,
\begin{equation}
\label{eqn::chi}
   \chi := \frac{2}{\sqrt{\kappa}} - \frac{\sqrt{\kappa}}{2},
\end{equation}
where $h$ is a GFF.  We remark that for $\kappa' > 4$, we have $\tfrac{\kappa'}{2}-4 > -2$ so $\SLE_{\kappa'}(\rho)$ with this range of $\rho$ values falls under the scope of \cite{MS_IMAG,MS_IMAG2,MS_IMAG3,MS_IMAG4}.  At $\rho = \tfrac{\kappa}{2}-4$, $\SLE_\kappa(\rho)$ for $\kappa \in (0,4)$ has a phase transition from the light cone regime described in this article to the loop-making/trunk regime studied by the authors together with Werner in \cite{cle_percolations}.   (In fact, as we will explain here and have also mentioned in \cite{cle_percolations}, the law of the range of an $\SLE_\kappa(\tfrac{\kappa}{2}-4)$ process is the same as the law of the range of an $\SLE_{\kappa'}(\tfrac{\kappa'}{2}-4)$ process, where $\kappa \in (0,4)$ and $\kappa' = 16/\kappa > 4$.)

See  Table~\ref{tab::rho_values} and Figure~\ref{fig::rho_kappa_chart} for a summary of the phases of $\SLE_\kappa(\rho)$.

Our first main result concerns continuity and transience.

\begin{theorem}
\label{thm::continuous}
The $\SLE_\kappa(\rho)$ processes for $\kappa \in (0,4)$, $\rho \in [\tfrac{\kappa}{2}-4,-2)$, and $\rho > -2-\tfrac{\kappa}{2}$ are almost surely continuous and transient.  That is, if $D \subseteq \C$ is a Jordan domain, $x,y \in \partial D$ are distinct, and $\eta \colon [0,\infty) \to D$ is an $\SLE_\kappa(\rho)$ in $D$ from $x$ to $y$ then $\eta$ is almost surely continuous and $\lim_{t \to \infty} \eta(t) = y$ almost surely.
\end{theorem}

The continuity of ordinary $\SLE$ was first proved by Rohde and Schramm in \cite{RS05}.  The main idea is to estimate the moments of the derivative of the reverse Loewner flow evaluated near the inverse image of the tip of the path.  By the Girsanov theorem, during a time interval in which $V_t \not = W_t$, the evolution of an $\SLE_\kappa(\rho)$ is absolutely continuous with respect to the evolution of ordinary $\SLE_\kappa$.  Consequently, the almost sure continuity of the process during such intervals of time can be easily derived from \cite{RS05}. From this one can see immediately that  $\SLE_\kappa(\rho)$ is a.s.\ continuous when $\rho \geq \tfrac{\kappa}{2}-2$ so that $\delta \geq 2$.  A more general statement is \cite[Theorem~1.3]{MS_IMAG}, which states that $\SLE_{\kappa}(\rho)$ is a.s.\ continuous for all  $\kappa$ and all $\rho > -2$. The idea of that proof is to extract the continuity from the non-boundary-intersecting case and a conditioning trick which involves multiple $\SLE$ paths coupled together using the GFF.  Theorem~\ref{thm::continuous} extends this further to the case that $\rho \geq \tfrac{\kappa}{2}-4$ and $\rho > -2-\tfrac{\kappa}{2}$.  Its proof is also based on GFF arguments, though the method is rather different than that of \cite[Theorem~1.3]{MS_IMAG}.  Continuity in the case that $\rho \in (-2-\tfrac{\kappa}{2},\tfrac{\kappa}{2}-4]$ was established in \cite{cle_percolations}, also using GFF based arguments.  Combining these works, we have $\SLE_\kappa(\rho)$ continuity for all of the regions shown in Figure~\ref{fig::rho_kappa_chart}.

\begin{table}
{\footnotesize
\begin{center}
\begin{tabular}{llllcc}
\toprule
\head{$\rho$} & \head{$\delta(\kappa,\rho)$}  & \head{$\dimH(\text{Range})$} & \head{Process type} & \head{Simple} & \head{Rev.}\\
\toprule

$(-\infty,-2-\tfrac{\kappa}{2}]$ & $(-\infty,0]$ & --- & Not defined & --- & ---\\

$(-2-\tfrac{\kappa}{2},\tfrac{\kappa}{2}-4]$ & $(0,2-\tfrac{4}{\kappa}] $ & $1+\tfrac{2}{\kappa}$ & Trunk plus loops & $\text{\sffamily X}$ & $\text{\sffamily X}$\\

$(\tfrac{\kappa}{2}-4,-2)$ & $(2-\tfrac{4}{\kappa},2)$ & $\tfrac{(\kappa-2(2+\rho))(\kappa+2(6+\rho))}{8\kappa}$ & Light cone & $\text{\sffamily X}$ & $\text{\sffamily X}$\\

$-2$ & $1$ & $1$ & $\partial$ tracing & $\checkmark$ & $\checkmark$\\

$(-2,\tfrac{\kappa}{2}-2)$ & $(1,2)$ & $1+\tfrac{\kappa}{8}$ & $\partial$ hitting & $\checkmark$ & $\checkmark$\\

$ [\tfrac{\kappa}{2}-2,\infty)$ & $[2,\infty)$ & $1+\tfrac{\kappa}{8}$ & $\partial$ avoiding & $\checkmark$ & $\checkmark$\\

\bottomrule
\end{tabular}
\end{center}}
\medskip
\caption{\label{tab::rho_values} Phases of~$\rho$ values and corresponding $\delta(\kappa,\rho)$ (driving Bessel process dimension) values for $\SLE_\kappa(\rho)$ processes with a single boundary force point of weight~$\rho$, assuming $\kappa \in (2,4)$. When $\kappa \in (0, 2]$, the phases are the same except that the second and third are replaced by a single ``light cone'' phase with $\rho \in (-2-\tfrac{\kappa}{2},-2)$ and $\delta \in (0,1)$. The symbol ``$\partial$'' should be translated as ``boundary'' and ``rev.'' stands for ``reversible.''  The statements in the reversible column are only applicable when the force point is located immediately to the left or to the right of the seed of the process.}
\end{table}

\begin{figure}[ht!]
\begin{center}
\includegraphics[scale=0.85]{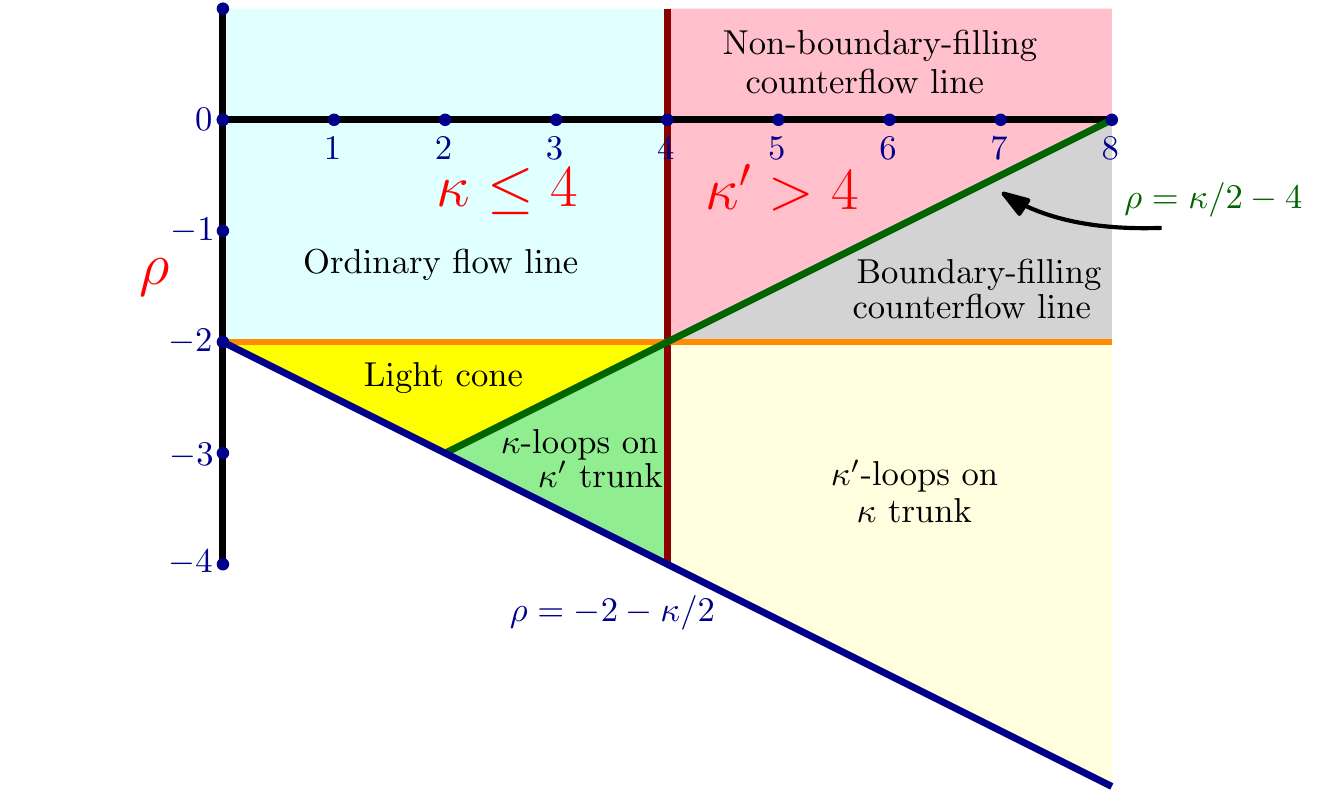}	
\end{center}
\caption{\label{fig::rho_kappa_chart} Phase diagram for the behavior of $\SLE_\kappa(\rho)$ for $\rho$ above the minimal value $-2-\tfrac{\kappa}{2}$ for which such a process is defined.  The present paper is focused on the light cone regime (yellow triangle) where $\kappa \in (0,4)$ and $\rho \in ( (-2-\tfrac{\kappa}{2}) \vee (\tfrac{\kappa}{2}-4),-2)$.  The other two $\rho < -2$ regimes are studied in \cite{cle_percolations} and the $\rho > -2$ cases are treated in \cite{MS_IMAG,MS_IMAG4}.}  
\end{figure}

Suppose that $D \subseteq \C$ is a Jordan domain, $x \in \partial D$, and $h$ is a GFF on~$D$ with given boundary conditions.  Fix angles $\theta_1 \leq \theta_2 \leq \theta_1 + \pi$.  The {\bf $\SLE_\kappa$ light cone} $\lightcone_x(\theta_1,\theta_2)$ of~$h$ starting from $x$ with angle range $[\theta_1,\theta_2]$ is a random set in~$D$ generated from the flow lines of $e^{i h / \chi}$ (hereafter, we will refer to these simply as ``flow lines of $h$'').  It is explicitly given by the closure of the set of points accessible by the flow lines of $h$ starting from $x$ with angles which are either rational and contained in $[\theta_1,\theta_2]$ or equal to~$\theta_1$ or~$\theta_2$ and which change angles a finite number of times and only at positive rational times.  These objects were first introduced in \cite{MS_IMAG}.  We call $\theta_2-\theta_1$ the {\bf opening angle} of $\lightcone_x(\theta_1,\theta_2)$.  For $\theta \in [0,\pi]$, we let $\lightcone_x(\theta) = \lightcone_x(-\tfrac{\theta}{2},\tfrac{\theta}{2})$.  It is shown in \cite[Theorem~1.4]{MS_IMAG} that a light cone with opening angle $\pi$ starting from $x$ is equal to the range of a form of $\SLE_{16/\kappa}$, which is called a {\em counterflow line} targeted at $x$.  More generally, if $A$ is a segment of $\partial D$, we let $\lightcone_A(\theta_1,\theta_2)$ be the set points accessible by flow lines of $h$ starting from a countable dense subset of $A$ with angles which are either rational and contained in $[\theta_1,\theta_2]$ or equal to $\theta_1$ or $\theta_2$ which change angles only a finite number of times and only at positive rational times.  Our next result states that $\lightcone_{\R_-}(0,\theta)$ for the intermediate values of $\theta \in (0,\pi)$ is equal to the range of an $\SLE_\kappa(\rho)$ process provided the boundary data of $h$ is chosen appropriately.

Let
\begin{equation}
\label{eqn::lambda}
\lambda := \frac{\pi}{\sqrt{\kappa}}.
\end{equation}

\begin{theorem}
\label{thm::coupling}
Fix $\kappa \in (0,4)$, $\rho \in [\tfrac{\kappa}{2}-4,-2)$ and $\rho > -2-\tfrac{\kappa}{2}$, and suppose that $h$ is a GFF on $\h$ whose boundary data is given by $-\lambda$ on $\R_-$ and $\lambda(1+\rho)$ on $\R_+$.  Let $\eta$ be an $\SLE_\kappa(\rho)$ process on $\h$ from $0$ to $\infty$ where its force point is located at $0^+$.  For each $t \geq 0$, let $K_t$ denote the closure of the complement of the unbounded connected component of $\h \setminus \eta([0,t])$, let $g_t \colon \h \setminus K_t \to \h$ be the unique conformal transformation with $\lim_{z \to \infty} |g_t(z) - z| = 0$, and let $(W,V)$ be the Loewner driving pair for $\eta$.  There exists a unique coupling of $h$ and $\eta$ such that the following is true.  For each $\eta$-stopping time $\tau$, the conditional law of
\[ h \circ g_\tau^{-1} - \chi \arg( g_\tau^{-1})'\]
given $\eta|_{[0,\tau]}$ is that of a GFF on $\h$ with boundary conditions given by
\[ h|_{(-\infty,W_\tau]} \equiv -\lambda,\quad h|_{(W_\tau,V_\tau]} \equiv \lambda, \quad\text{and}\quad h|_{(V_\tau,\infty)} \equiv \lambda(1+\rho).\]
Moreover, in the coupling $(h,\eta)$, $\eta$ is almost surely determined by $h$.  Finally, let 
\begin{equation}
\label{eqn::lightcone_angle}
 \theta = \theta_\rho = \pi\left(\frac{\rho+2}{\kappa/2-2} \right).
\end{equation}
Then the range of $\eta$ is almost surely equal to $\lightcone_{\R_-}(0,\theta)$.
\end{theorem}

We remark that the existence statement in Theorem~\ref{thm::coupling} takes the same form as that for $\SLE_\kappa(\rho)$ when $\rho > -2$, e.g.,\ \cite[Theorem~1.1]{MS_IMAG}.  The proof that we give here, however, is quite different.  The difference between the different regimes of $\rho$ values is in the way that the coupling is interpreted.  In particular, we interpret the process when $\rho > -2$ as being a flow line of the (formal) vector field $e^{i h /\chi}$ (see the introductions to \cite{SHE_WELD, MS_IMAG} for further explanation) while we interpret the process when $\rho \in [\tfrac{\kappa}{2}-4,-2)$ as an ordered light cone of flow lines of $e^{i h / \chi}$.  The method that we use to prove existence in Theorem~\ref{thm::coupling} is also very different from the existence proof given in \cite{She_SLE_lectures,SHE_WELD,SchrammShe10,DUB_PART} for $\rho > -2$.  Indeed, in these works existence is shown by proving that a sample of the GFF can be produced by first sampling the path according to its marginal distribution and then sampling a GFF on the complement of the range of the path with appropriate boundary conditions.  That the marginal law of the field is a GFF is proved using tools from stochastic calculus.  In the present work, we will use the flow line interaction theory from \cite{MS_IMAG,MS_IMAG2,MS_IMAG3,MS_IMAG4} and the local set theory from \cite{SchrammShe10} to show directly that the path which arises by visiting the points of a light cone with a particular order evolves as an $\SLE_\kappa(\rho)$.  The final statement of Theorem~\ref{thm::coupling} generalizes \cite[Theorem~1.4]{MS_IMAG} to the setting of $\SLE_\kappa(\rho)$ for $\rho \in [\tfrac{\kappa}{2}-4,-2)$ and $\rho > -2-\tfrac{\kappa}{2}$.  In the case of the former, the result followed by studying the manner in which flow and counterflow lines coupled together with the GFF interact with each other.  The proof of Theorem~\ref{thm::coupling} is different.  We will extract the latter from the corresponding result for $\SLE_{\kappa'}(\rho)$ processes with $\rho > -2$ proved in \cite[Theorem~1.3]{MS_IMAG}.

Let $\dimH(A)$ denote the Hausdorff dimension of a set $A$.  The almost sure value of $\dimH(\lightcone_x(\theta))$ is computed in \cite[Theorem~1.1]{LIGHTCONE_DIMENSION}.  

Combining this with Theorem~\ref{thm::coupling} gives that if $\eta$ is an $\SLE_\kappa(\rho)$ process with $\kappa \in (0,4)$, $\rho \in [\tfrac{\kappa}{2}-4,-2)$, and $\rho > -2-\tfrac{\kappa}{2}$, then 
\begin{equation}
\label{eqn::dimension} \dimH(\eta) = \frac{(\kappa-2(2+\rho))(\kappa+2(6+\rho))}{8\kappa} \quad\text{almost surely}.
\end{equation}
This result is stated as \cite[Theorem~1.2]{LIGHTCONE_DIMENSION}.

The decomposition of the range of $\SLE_\kappa(\rho)$ into a light cone of angle-varying flow lines is related to the notion of duality for $\SLE_\kappa$.  The principle of duality states that the outer boundary of an $\SLE_{\kappa'}$ process can be described by a form of $\SLE_\kappa$ for $\kappa \in (0,4)$ and $\kappa'=16/\kappa \in (4,\infty)$, \cite{ZHAN_DUALITY_1,ZHAN_DUALITY_2,DUB_DUAL,MS_IMAG,MS_IMAG4}.  Since the range of an $\SLE_{\kappa'}$ process can be described in terms of a light cone with opening angle $\pi$, it thus follows from Theorem~\ref{thm::coupling} that the law of the range of an $\SLE_\kappa(\tfrac{\kappa}{2}-4)$ is the same as that of a form of $\SLE_{\kappa'}$ (specifically, an $\SLE_{\kappa'}(\tfrac{\kappa'}{2}-4)$).  It turns out, however, that the two processes visit the points in their range using a different order.  This is explained in more detail in Section~\ref{sec::limiting_cases} as well as in \cite{cle_percolations}.  Our final result is the continuity of the law of an $\SLE_\kappa(\rho)$ process as a function of $\rho$ with $\rho$ in the light cone regime.

\begin{theorem}
\label{thm::interpolation}
Fix $\kappa \in (0,4)$, let $D \subseteq \C$ be a bounded Jordan domain, and fix $x,y \in \partial D$ distinct.  The law of the trajectory of an $\SLE_\kappa(\rho)$ process from $x$ to $y$ in $D$ is continuous with respect to the weak topology induced by the topology of uniform convergence modulo time parameterization as $\rho$ varies between $(-2-\tfrac{\kappa}{2})\vee(\tfrac{\kappa}{2}-4)$ and $-2$.
\end{theorem}

\subsection*{Outline}

The remainder of this article is structured as follows.  In Section~\ref{sec::preliminaries}, we will give some preliminaries. In Section~\ref{sec::gff_couplings} we will prove Theorem~\ref{thm::coupling} and then use it to derive Theorem~\ref{thm::continuous} and Theorem~\ref{thm::interpolation}.  Finally, in Section~\ref{sec::limiting_cases} we will explain why the law of the range of an $\SLE_\kappa(\tfrac{\kappa}{2}-4)$ process for $\kappa \in (2,4)$, which is at the boundary of the light cone regime, is equal to the law of a range of an $\SLE_{\kappa'}(\tfrac{\kappa'}{2}-4)$ process, but the processes visit their range in a different order.

\section{Preliminaries}
\label{sec::preliminaries}

In this section, we are going to give an overview of the chordal $\SLE_\kappa(\rho)$ processes, focusing on the particular case that $\rho \in (-2-\tfrac{\kappa}{2},-2)$, as well as summarize some of the basics of imaginary geometry \cite{MS_IMAG,MS_IMAG2,MS_IMAG3,MS_IMAG4} which is relevant for this work.

\subsection{$\SLE_\kappa(\rho)$ processes}
\label{subsec::sle}

In this subsection, we are going to give an overview of the $\SLE_\kappa(\rho)$ processes.  These are variants of $\SLE$ first introduced in \cite[Section~8.3]{LSW_RESTRICTION}.  They are defined in the same way as ordinary $\SLE$, except they are driven by a multiple of a Bessel process in place of a Brownian motion.  The treatment that we give here will parallel that from \cite[Section~3.2 and Section~3.3]{SHE_CLE}.

For the convenience of the reader, we will now review a few basic facts about Bessel processes.  (We refer the reader to \cite[Chapter~11]{RY04} for a more in-depth introduction.)  The starting point for the construction of the law of a \emph{Bessel process of dimension $\delta$} ($\BES^\delta$) is the so-called \emph{square Bessel process of dimension $\delta$} ($\BESQ^\delta$).  For a fixed value of $\delta \in \R$, the law of a $\BESQ^\delta$ is described in terms of the SDE
\begin{equation}
\label{eqn::besq}
dZ_t = \delta dt + 2\sqrt{Z_t} dB_t ,\quad Z_0 = z_0 > 0,
\end{equation}
where $B$ is a standard Brownian motion.  Standard results for SDEs imply that there is a unique strong solution to~\eqref{eqn::besq}, at least up until the first time that $Z$ hits $0$.  When $\delta > 0$, there in fact exists a unique strong solution for all $t \geq 0$ which is non-negative for all times.

A process $X$ has the $\BES^\delta$ law if it admits the expression $X = \sqrt{Z}$ where $Z$ is a $\BESQ^\delta$.  It\^o's formula implies that $X$ solves the SDE
\begin{equation}
\label{eqn::bes}
dX_t = \frac{\delta-1}{2} \cdot \frac{1}{X_t} dt + dB_t,\quad X_0 = x_0,
\end{equation}
at least up until the first time that $X$ hits $0$.  Using that $X_t^{2-\delta}$ is a continuous local martingale, it is straightforward to check that a $\BES^\delta$ process almost surely hits $0$ if $\delta < 2$ and almost surely does not hit $0$ if $\delta \geq 2$.  When $\delta > 1$, a $\BES^\delta$ process solves~\eqref{eqn::bes} in integrated form for all $t \geq 0$, even when it is bouncing off $0$.  In particular, such processes are semimartingales.  A $\BES^1$ process $X$ is equal in distribution to $|B|$ where $B$ is a standard Brownian motion, hence in this case, by the It\^o-Tanaka formula, $X$ solves a version of~\eqref{eqn::bes} with an extra correction coming from the local time of $X$ at $0$.  Thus $\BES^1$ processes are also semimartingales.  However, $X_t^{-1}$ is not integrable in this case.  When $\delta \in (0,1)$, it turns out that a $\BES^\delta$ process is not a semimartingale.  In order to make sense of it as a solution to~\eqref{eqn::bes} in integrated form, one needs to make a so-called principal value correction.  Namely, $X$ satisfies the integral equation
\begin{equation}
\label{eqn::bes_pv}
 X_t = x_0 + \frac{\delta-1}{2} \PV \int_0^t \frac{1}{X_s} ds + B_t.
\end{equation}
As explained in \cite[Chapter~11]{RY04}, the principal value correction can be understood in terms of an integral of the local time process of $X$ at $0$.  We will not discuss the details of this here since the properties and definition of the principal value correction will not play much of a role in this work.

The Bessel processes that we have discussed so far are always non-negative.  We remark that it is also natural in certain contexts to consider Bessel processes which can take on both positive and negative values.  These processes can be constructed by starting off with a Bessel process which is always non-negative and then assigning a random sign to each excursion the process makes from $0$ as a result of the flip of an independent coin toss.  These processes give rise to so-called side-swapping $\SLE_\kappa(\rho)$ processes, which we will not discuss in the present article.

As mentioned just above, the $\BES^\delta$ processes are the starting point for constructing the so-called $\SLE_\kappa(\rho)$ processes.  Fix $\kappa > 0$, $\rho > -2-\tfrac{\kappa}{2}$, and let
\[ \delta = 1+\frac{2(\rho+2)}{\kappa}.\]
Note that $\delta > 0$.  Let $X_t$ be a $\BES^\delta$ and let 
\begin{align*}
	V_t = \frac{2}{\sqrt{\kappa}} \PV \int_0^t \frac{1}{X_s} ds \quad\text{and}\quad W_t = V_t - \sqrt{\kappa} X_t.
\end{align*}
Then the chordal Loewner $(g_t)$ chain driven by $W$, i.e., the solution to the ODE
\[ \partial_t g_t(z) = \frac{2}{g_t(z) - W_t},\quad g_0(z) = z,\]
is an $\SLE_\kappa(\rho)$ process.  The point $g_t^{-1}(V_t)$ gives the location of the so-called \emph{force point} of the $\SLE_\kappa(\rho)$ process at time $t$.

Let us make a few comments about this definition.  In the case that $\rho > -2$ so that $\delta > 1$, the principal value integral is the same as the usual integral.  This implies that $V_t$ is equal to the image under $g_t$ of the rightmost intersection point of the corresponding hull at time~$t$ with~$\R$.  Equivalently, the force point at each time $t$ is located at the rightmost intersection of the hull with~$\R$.  The continuity of the processes in this case were established in \cite{MS_IMAG}, building off the continuity of $\SLE_\kappa$ proved in \cite{RS05}.  In the case that $\rho \in (-2-\tfrac{\kappa}{2},-2)$ so that $\delta \in (0,1)$, the force point of an $\SLE_\kappa(\rho)$ process \emph{does not} stay in $\R$, as a consequence of the principal value correction which is necessary for its definition.  In the case that $\rho \in (-2-\tfrac{\kappa}{2},\tfrac{\kappa}{2}-4]$ and $\kappa \in (2,4)$, the continuity of these processes was proved in \cite{cle_percolations} using couplings of these processes with the GFF and as a consequence of the continuity of so-called space-filling $\SLE$ established in \cite{MS_IMAG4}.  In the present work, we will prove the continuity of these processes for $\rho \in ((-2-\tfrac{\kappa}{2}) \vee (\tfrac{\kappa}{2}-4),-2)$, also using the GFF and the continuity of space-filling $\SLE$, thus covering all possible cases.

The $\SLE_\kappa(\rho)$ processes with $\rho \in (-2-\tfrac{\kappa}{2},-2)$ admit certain approximations which are described in \cite[Section~6]{SHE_CLE}.  The reader might find the description contained there helpful for understanding why the principal value correction leads to the force point of the process not always being on the domain boundary.

We finish this subsection by collecting the following technical result, which we will use in Section~\ref{sec::gff_couplings} in conjunction with \cite[Theorem~2.4]{MS_IMAG} to construct couplings between the $\SLE_\kappa(\rho)$ processes with $\rho \in ((-2-\tfrac{\kappa}{2}) \vee (\tfrac{\kappa}{2}-4),-2)$ and the GFF.

\begin{proposition}
\label{prop::bessel_pv}
Suppose that $X$ is a $\BES^\delta$ with $\delta \in (0,1)$ and that $U$ is a continuous process coupled with $X$ such that $(X,U)$ is strong Markov and possesses the following properties:
\begin{enumerate}[(i)]
\item\label{it::brownian_scaling} $(X,U)$ satisfies Brownian scaling: $t \mapsto (X_{\alpha t}, U_{\alpha t}) \stackrel{d}{=} t \mapsto \sqrt{\alpha}(X_t,U_t)$ for each $\alpha \geq 0$,
\item\label{it::derivative} for each $t \geq 0$ such that $X_t \neq 0$, we have $\frac{d}{dt} U_t = X_t^{-1}$,
\item\label{it::plus_minus_infinity}  $\limsup_{t \to \infty} U_t = \infty$ and $\liminf_{t \to \infty} U_t = -\infty$ almost surely, and
\item\label{it::independent} if $\tau$ is any stopping time for $X$ such that $X_{\tau} = 0$ and $t \geq 0$, then the law of $U_{t+\tau} - U_{\tau}$ is independent of $\sigma( (X_s,U_s) : s \leq \tau)$.
\end{enumerate}
Then
\begin{equation}
\label{eqn::o_pv}
 U_t = \PV \int_0^t \frac{1}{X_s} ds \quad\text{for all}\quad t \geq 0 \quad\text{almost surely}.
\end{equation}
\end{proposition}
\begin{proof}
The choice of~$U$ given by \eqref{eqn::o_pv} satisfies the hypotheses of the proposition, so it suffices to show that it is the only choice which satisfies the hypotheses.  Suppose that $U$, $\wt{U}$ are two processes which satisfy the properties above and are coupled with~$X$ such that~$U$, $\wt{U}$ are independent given $X$ and let $\ol{U} = U - \wt{U}$.  Let~$\ell$ denote the local time for~$X$ at~$0$ and, for each $s \geq 0$, let $t(s) = \inf\{t \geq 0 : \ell_t = s\}$.  Note that $\frac{d}{dt} \ol{U}_t = 0$ for $t \geq 0$ such that $X_t \neq 0$.  This implies that $s \mapsto \ol{U}_{t(s)}$ is a continuous process.  Indeed, if $u \uparrow s$ then $t(u) \uparrow t(s)$ so that $\ol{U}_{t(u)} \to \ol{U}_{t(s)}$.  Let $r$ be the limit of $t(u)$ as $u \downarrow s$.  Then $\ell$ is constant on $(t(s),r)$ hence $\ol{U}_{t(s)} = \ol{U}_r$ and, since $\ol{U}$ is continuous, $\lim_{u \downarrow s} \ol{U}_{t(u)} = \ol{U}_r$.  Therefore $\lim_{u \to s} \ol{U}_{t(u)} = \ol{U}_{t(s)}$, which proves the desired continuity.  By the strong Markov property and \eqref{it::independent}, we also know that $\ol{U}_{t(s)}$ has stationary, independent increments.  This implies that there exists a standard Brownian motion $B$ and constants $c_1,c_2 \in \R$ such that $\ol{U}_{t(s)} = c_1 B_s + c_2 s$.  Equivalently, $\ol{U}_t = c_1 B_{\ell_t} + c_2 \ell_t$.  Since $\ol{U}$, $\ell$, and $B$ all satisfy Brownian scaling, it is easy to see that $c_1 = 0$.  That $c_2 = 0$ then follows from~\eqref{it::plus_minus_infinity} since $\ell_t \to \infty$ almost surely as $t \to \infty$ because $\delta \in (0,1)$.  This implies that there exists at most one process $U$ which satisfies the hypotheses of the proposition.
\end{proof}

\subsection{Imaginary geometry review}
\label{subsec::sle_gff}

We assume in this work that the reader is familiar with the GFF and with imaginary geometry.  We direct the reader to \cite{SHE06} for a more in depth introduction to the GFF and to \cite{MS_IMAG} for a basic introduction to imaginary geometry.  In the present section, we will remind the reader of a few facts which are established in \cite{MS_IMAG,MS_IMAG4} about the manner in which flow lines interact with each other and the definition of space-filling $\SLE$.

We begin with a review of the coupling of chordal $\SLE_\kappa(\rho)$ with $\rho > -2$ with the GFF.  Throughout, we assume that $\kappa \in (0,4)$, $\kappa'=16/\kappa \in (4,\infty)$, and let
\begin{align*}
\chi = \frac{2}{\sqrt{\kappa}}- \frac{\sqrt{\kappa}}{2},\quad \lambda = \frac{\pi}{\sqrt{\kappa}},\quad\text{and}\quad \lambda' = \frac{\pi}{\sqrt{\kappa'}} = \lambda - \frac{\pi}{2} \chi. 	
\end{align*}
(These are the same values as in~\eqref{eqn::chi} and~\eqref{eqn::lambda}.)

\begin{figure}[ht!]
\begin{center}
\includegraphics[scale=0.85]{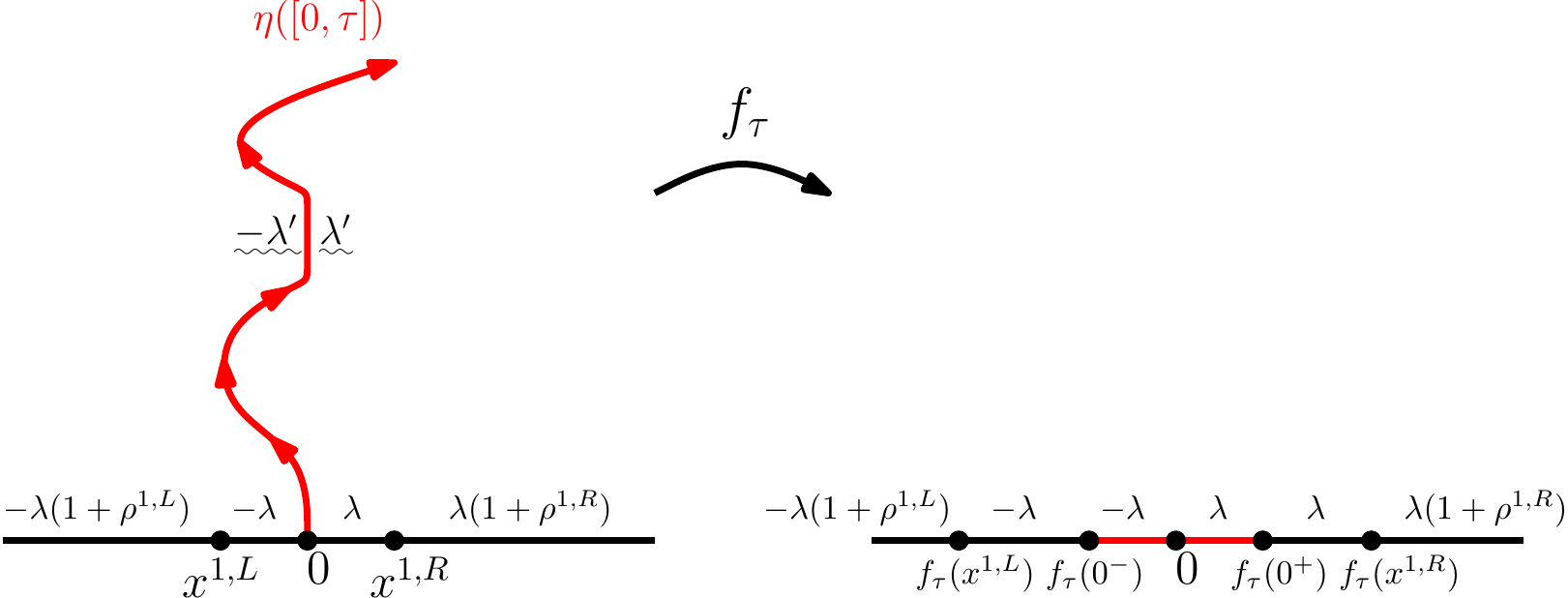}
\caption{\label{fig::conditional_boundary_data}  Suppose that $h$ is a GFF on $\h$ with the boundary data depicted above.  Then the flow line $\eta$ of $h$ starting from $0$ is an $\SLE_\kappa(\ul{\rho}^L;\ul{\rho}^R)$ curve in $\h$ where $|\ul{\rho}^L| = |\ul{\rho}^R| = 1$.  For any $\eta$ stopping time $\tau$, the law of $h$ given $\eta|_{[0,\tau]}$ is equal in distribution to a GFF on $\h \setminus \eta([0,\tau])$ with the boundary data depicted above (the notation $\uwave{a}$ is explained in \cite[Figure~1.10]{MS_IMAG}).  It is also possible to couple $\eta' \sim\SLE_{\kappa'}(\ul{\rho}^L;\ul{\rho}^R)$ for $\kappa' > 4$ with $h$ and the boundary data takes on the same form (with $-\lambda'$, $\lambda' := \frac{\pi}{\sqrt \kappa'}$, in place of $\lambda := \frac{\pi}{\sqrt \kappa}$).  The difference is in the interpretation.  The (almost surely self-intersecting) path $\eta'$ is not a flow line of $h$, but for each $\eta'$ stopping time $\tau'$ the left and right {\em boundaries} of $\eta'([0,\tau'])$ are $\SLE_{\kappa}$ flow lines, where $\kappa=16/\kappa'$, angled in opposite directions.  The union of the left boundaries --- over a collection of $\tau'$ values --- is a tree of merging flow lines, while the union of the right boundaries is a corresponding dual tree whose branches do not cross those of the tree.}
\end{center}
\end{figure}

We suppose that $\rho > -2$ is fixed and that~$h$ is an instance of the GFF on~$\h$ with boundary conditions $\lambda(1+\rho)$ (resp.\ $-\lambda$) on $\R_+$ (resp.\ $\R_-$).  Then it is shown in \cite[Theorem~1.1]{MS_IMAG} that there exists a unique coupling $(h,\eta)$ of~$h$ with an $\SLE_\kappa(\rho)$ process~$\eta$ in~$\h$ from~$0$ to~$\infty$ with a single boundary force point located at~$0^+$ such that the following is true.  Suppose that $(W,V)$ is the Loewner driving pair for $\eta$, $(g_t)$ the corresponding family of conformal maps, and that $\tau$ is an $\eta$-stopping time.  Then the conditional law of $h \circ g_\tau^{-1} - \chi \arg (g_\tau^{-1})'$ given $\eta|_{[0,\tau]}$ is that of a GFF on $\h$ with boundary conditions given by
\[ h|_{(-\infty,W_\tau]} \equiv -\lambda,\quad h|_{(W_\tau,V_\tau]} \equiv \lambda,\quad\text{and}\quad h|_{(V_\tau,\infty)} \equiv \lambda (1+\rho).\]
Equivalently, the conditional law of $h$ given $\eta|_{[0,\tau]}$ restricted to the unbounded component $\h_\tau$ of $\h \setminus \eta([0,\tau])$ is that of a GFF with the same boundary conditions as $h$ on $\partial \h \cap \partial \h_\tau$ and with boundary conditions which are given by $-\lambda'$ (resp.\ $\lambda'$) plus $\chi$ times the winding of $\eta$ along $\eta|_{[0,\tau]}$.  Since $\eta$ is not a smooth curve, its winding is not well-defined along the curve itself, however the harmonic extension of its winding is defined.  We will indicate this type of boundary data in the figures that follow using the notation introduced in \cite[Figure~1.10]{MS_IMAG}.  It is shown in \cite[Theorem~1.2]{MS_IMAG} that $\eta$ is almost surely determined by $h$, which is not an obvious statement from how the coupling is constructed.

The path $\eta$ has the interpretation of being a flow line of the vector field $e^{i h/\chi}$.  Similar statements hold in the presence of more general piecewise constant boundary data.  In the more general setting, the flow line is an $\SLE_\kappa(\ul{\rho})$ process where the number of force points is equal to the number of jumps in the boundary data for $h$.  See Figure~\ref{fig::conditional_boundary_data} for an illustration in the case of two force points.  Similar statements also hold for the existence of a unique coupling of an $\SLE_{\kappa'}$ process $\eta'$ with the GFF, except the interpretation is different.  We refer to $\eta'$ \emph{counterflow line} of $h$ because an $\SLE_{\kappa'}$ process can be realized as a light cone of flow lines which travel in the opposite direction of $\eta'$.

We refer to a path coupled as a flow line with $h+\theta \chi$ as the flow line of~$h$ with angle~$\theta$.  This is because such a path has the interpretation of being the flow line of the vector field $e^{i (h \chi + \theta)}$, i.e., the field which arises by taking all of the arrows in $e^{i h/ \chi}$ and then rotating them by the angle $\theta$.  The manner in which flow lines with different angles interact is established in \cite[Theorem~1.5]{MS_IMAG} as well as \cite[Theorem~1.7]{MS_IMAG4}.  Specifically, if $\eta_{\theta_1}$ (resp.\ $\eta_{\theta_2})$ are the flow lines of a GFF $h$ on $\h$ starting from $x_1 \leq x_2$, then the following holds.  If $\theta_1 > \theta_2$, then $\eta_{\theta_1}$ stays to the left of (but may bounce off) $\eta_{\theta_2}$.  If $\theta_1 = \theta_2$, then $\eta_{\theta_1}$ and $\eta_{\theta_2}$ merge upon intersecting and do not subsequently separate.  Finally, if $\theta_2 - \pi < \theta_1 < \theta_2$, then $\eta_{\theta_1}$ and $\eta_{\theta_2}$ cross upon intersecting for the first time.  After crossing, the paths may continue to bounce off each other but do not cross again.

One can also consider couplings of $\SLE$ with the GFF on domains other than~$\h$.  Specifically, suppose that $D \subseteq \C$ is a simply connected domain and $x,y \in \partial D$ are distinct.  Then to construct a coupling an $\SLE_\kappa(\ul{\rho})$ process $\eta$ in $D$ from $x$ to $y$ with a GFF $h$ on $D$, one starts with such a coupling $(\wt{h},\wt{\eta})$ on $\h$ and then takes
\begin{equation}
\label{eqn::change_coordinates}	
h = \wt{h} \circ \varphi^{-1} - \chi \arg (\varphi^{-1})' \quad\text{and}\quad \eta = \varphi(\wt{\eta})
\end{equation}
where $\varphi \colon \h \to D$ is a conformal transformation which takes $0$ to $x$ and $\infty$ to $y$.  We note that this change of coordinates formula is the same as the one which corresponds to the flow lines of $e^{i h / \chi}$ in the setting that $h$ is a continuous function.

Flow lines of the GFF starting from interior points were constructed and studied in \cite{MS_IMAG4}.  The interaction rules for these paths are the same as in the setting of paths which start on the domain boundary; see \cite[Theorem~1.7]{MS_IMAG4}.  In \cite{MS_IMAG4}, these paths were used to construct so-called \emph{space-filling $\SLE_{\kappa'}$}, which is a form of ordinary $\SLE_{\kappa'}$ except whenever it cuts off a component, it branches in and fills it up before continuing.  Specifically, we suppose that $h$ is a GFF on $\h$ with boundary conditions given by~$\lambda'$ (resp.\ $-\lambda'$) on~$\R_-$ (resp.\ $\R_+$).  (These are the boundary conditions so that the counterflow line of~$h$ from~$0$ to $\infty$ is an $\SLE_{\kappa'}$ process.)  Fix a countable dense set $(w_n)$ in~$\h$ and, for each $n$, we let~$\eta_n$ be the flow line of~$h$ starting from~$w_n$ with angle $\pi/2$.  Then we say that $w_n$ \emph{comes before} $w_m$ if~$\eta_n$ merges with~$\eta_m$ on its left side (see, e.g., \cite[Figure~1.16]{MS_IMAG4}).  This defines an ordering on the $(w_n)$ and space-filling $\SLE_{\kappa'}$ is a non-crossing random path which fills all of~$\h$ and visits the $(w_n)$ in this order.

It turns out that if we target a space-filling $\SLE_{\kappa'}$ process at a given point $z$ (i.e., parameterize it according to capacity as seen from that point), then we obtain exactly the counterflow line of the GFF targeted at $z$.  Therefore the aforementioned ordering also determines the order in which a counterflow line visits the points in its range.

The space-filling $\SLE_{\kappa'}(\ul{\rho})$ processes are defined in an analogous way by starting with a GFF with different boundary data.  One can similarly order space using flow lines of any given angle $\theta$ rather than the angle $\pi/2$ and obtain a continuous, space-filling path.

\section{GFF couplings}
\label{sec::gff_couplings}

In this section, we are going to prove Theorem~\ref{thm::continuous} and Theorem~\ref{thm::coupling} simultaneously and then explain how to extract Theorem~\ref{thm::interpolation} from these results.  We will begin in Section~\ref{subsec::lightcones} by proving several results about the structure of the complementary components (``pockets'') of light cones and then in Section~\ref{subsec::explorations} we will explain how we can use an $\SLE_{\kappa'}$, $\kappa'=16/\kappa \in (4,\infty)$, counterflow line to generate a continuous path which explores the range of a light cone.  In both of these sections, we will restrict ourselves to the case in which the light cone starts from a single boundary point (rather than a continuum) so that we can work in a unified framework.  We will then explain in Section~\ref{subsec::law} that these results also hold in the setting in which the light cone starts from a continuum of boundary points using a conditioning argument and then make the connection to $\SLE_\kappa(\rho)$ processes with $\rho \in [\tfrac{\kappa}{2}-4,-2)$ and $\rho > -2-\tfrac{\kappa}{2}$.

Throughout, unless explicitly stated otherwise, we shall assume that~$h$ is a GFF on~$\D$ which is given by a conformal coordinate change as in~\eqref{eqn::change_coordinates} of a GFF on~$\h$ with piecewise constant boundary data which changes values at most a finite number of times.  The reason for this is that it will be more convenient to work on a bounded Jordan domain rather than~$\h$ because then $\SLE_{\kappa'}$ is uniformly continuous.  We also let
\begin{equation}
\label{eqn::critical_angle}
\theta_c = \frac{\pi \kappa}{4-\kappa}.
\end{equation}
This is the so-called {\bf critical angle} --- the angle difference below which GFF flow lines can intersect each other and at or above which they cannot (see \cite[Theorem~1.5]{MS_IMAG} and \cite[Theorem~1.7]{MS_IMAG4}).  It is shown in \cite{LIGHTCONE_DIMENSION} that the almost sure dimension of a light cone with opening angle $\theta \in [0,\theta_c \wedge \pi)$ is contained in $[1,2)$ and that the dimension is equal to $2$ for $\theta \in [\theta_c \wedge \pi,\pi]$.  Note that $\theta_c \leq \pi$ if and only if $\kappa \leq 2$, which is closely connected with the fact that ordinary $\SLE_{\kappa'}$ is space-filling if and only if $\kappa' \geq 8$ \cite{RS05}.

\subsection{Pocket structure}
\label{subsec::lightcones}

Fix $\theta_1 \leq \theta_2$ with $\theta_2 - \theta_1 \leq \pi$.  For each $n \in \N$, let $\lightcone_n(\theta_1,\theta_2)$ be the closure of the set of points accessible by angle-varying flow lines of $h$ starting from $-i$ which travel either with angle $\theta_1$ or $\theta_2$, change directions at most $n$ times, and only change directions at positive rational times.  The {\bf light cone} $\lightcone(\theta_1,\theta_2) = \ol{\cup_n \lightcone_n(\theta_1,\theta_2)}$ of $h$ (starting from~$-i$) with angle range $[\theta_1,\theta_2]$ is the closure of the set of points accessible by flow lines of $h$ starting from~$-i$ with angle-varying trajectories with angle either equal to~$\theta_1$~or $\theta_2$ and which change directions a finite number of times and only at positive rational times.  Note that this definition is slightly different than that given in the introduction because we only allow the paths to travel with the extremal angles~$\theta_1$ and~$\theta_2$ (and do not allow the intermediate angles).  This definition will be more convenient for us to work with and we will shortly show that it and the one given in the introduction almost surely agree.  For $\theta \in [0,\pi]$, we also let $\lightcone(\theta) = \lightcone(-\tfrac{\theta}{2},\tfrac{\theta}{2})$.

\begin{figure}[ht!]
\begin{center}
\includegraphics[scale=0.85]{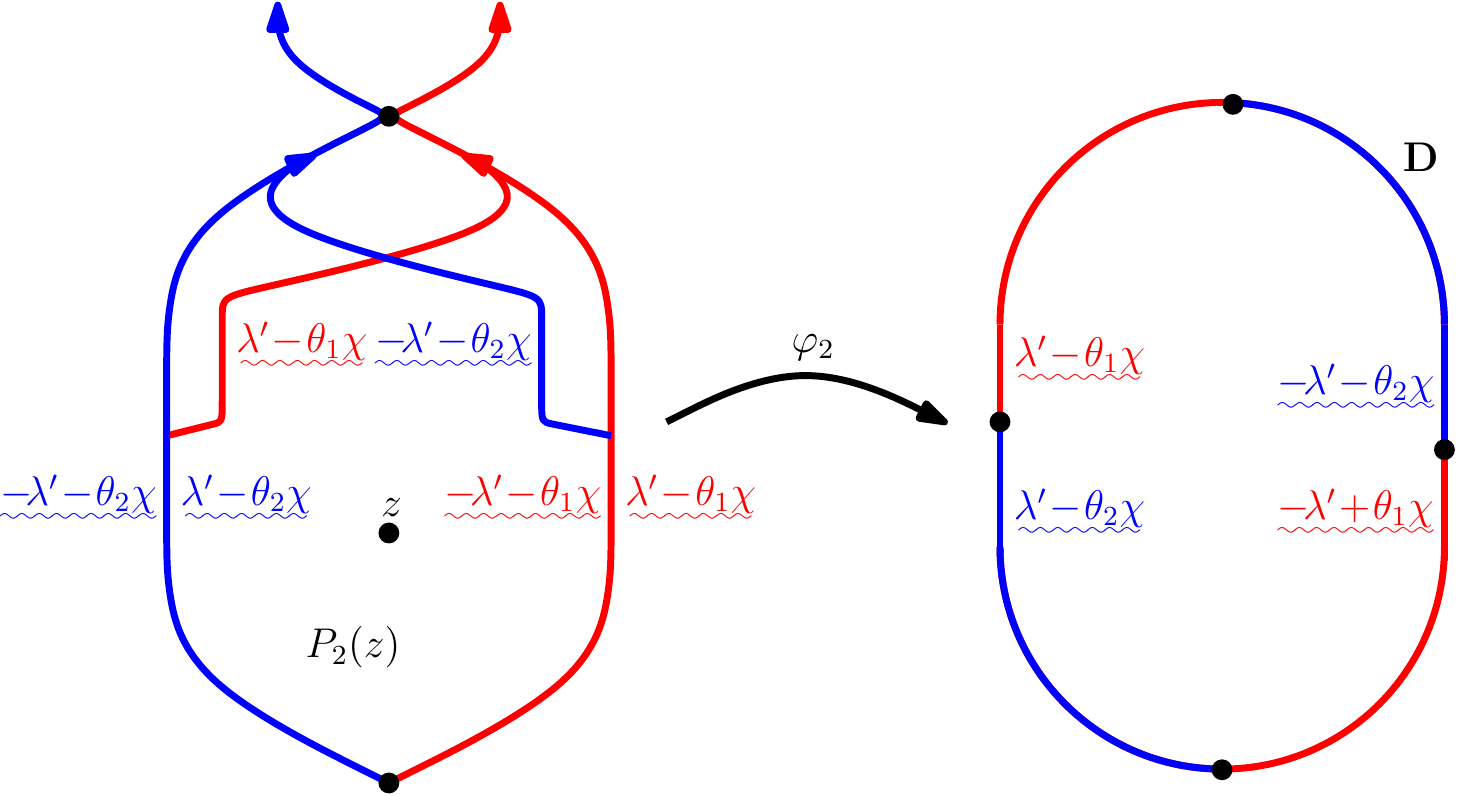}
\end{center}
\caption{\label{fig::lightcone_pocket_finite} Shown on the left is the pocket $P_2(z)$ of $\lightcone_2(\theta_1,\theta_2)$ containing~$z$ on the event that $P_2(z)$ separates $z$ from $\partial \D$.  We let $\varphi_2 \colon P_2(z) \to \D$ be the unique conformal transformation with $\varphi_2(z) = 0$ and $\varphi_2'(z) > 0$.  Shown on the right is the boundary data of the GFF $h \circ \varphi^{-1} - \chi \arg (\varphi^{-1})'$ on $\partial \D$.  The reason that $\D$ on the right side appears not to be perfectly round is so that we can use our notation to label the boundary data.}
\end{figure}

For each $z \in \D$ and $n \in \N$, let $\pocket{n}{z}$ be the complementary component of $\lightcone_n(\theta_1,\theta_2)$ which contains $z$ and let $\pocket{z}$ be the complementary component of $\lightcone(\theta_1,\theta_2)$ which contains $z$.  Throughout, we will refer to such complementary components as (complementary) {\bf pockets} of $\lightcone(\theta_1,\theta_2)$.  We are next going to describe the boundary data of $h$ given $\lightcone(\theta_1,\theta_2)$ on $\partial \pocket{z}$.  It is a consequence of the main result of \cite{LIGHTCONE_DIMENSION} that $\pocket{z} \neq \emptyset$ almost surely provided $\theta_2 - \theta_1 < \theta_c$ and $\theta_2 - \theta_1 \leq \pi$.

\begin{lemma}
\label{lem::form_pockets}
Suppose that $\theta_2 - \theta_1 < \theta_c$ and $\theta_2 - \theta_1 \leq \pi$.  Fix $z \in \D$ and assume that the event $E(z)$ that $\lightcone(\theta_1,\theta_2)$ disconnects $z$ from $\partial \D$ has positive probability.  On $E(z)$, let $\varphi \colon \pocket{z} \to \D$ be the unique conformal transformation with $\varphi(z) = 0$ and $\varphi'(z) > 0$.  Then the boundary data for $\wt{h} = h \circ \varphi^{-1} - \chi \arg (\varphi^{-1})'$ is as described in the left side of Figure~\ref{fig::lightcone_pocket}.  In particular, there exists two distinct marked points $\openL,\closeL \in \partial \pocket{z}$ such that the boundary behavior of $h$ along the clockwise (resp.\ counterclockwise) boundary segment $\side{1}{z}$ (resp.\ $\side{2}{z}$) of $\partial \pocket{z}$ from $\openL$ to $\closeL$ is the same as that of the right (resp.\ left) side of a flow line with angle $\theta_1$ (resp.\ $\theta_2$).
\end{lemma}
\begin{proof}
Assume that we are working on $E(z)$.  Then there exists $n_0 \in \N$ such that $\lightcone_n(\theta_1,\theta_2)$ separates~$z$ from~$\partial \D$ for all $n \geq n_0$.  For each $n \geq n_0$, let $\varphi_n \colon \pocket{n}{z} \to \D$ be the unique conformal transformation with $\varphi_n(z) = 0$ and $\varphi_n'(z) > 0$.  Let $\wt{h}_n = h \circ \varphi_n^{-1} - \chi \arg (\varphi_n^{-1})'$ be the GFF on $\D$ given by conformally mapping $\pocket{n}{z}$ to~$\D$ using~$\varphi_n$ and applying the coordinate change formula~\eqref{eqn::change_coordinates}.  As shown in Figure~\ref{fig::lightcone_pocket_finite} (in the case that $n=2$), the boundary data for~$\wt{h}_n$ has four (possibly degenerate) marked points.  These divide $\partial \D$ into the images~$L_n^{\theta_1}$ and~$L_n^{\theta_2}$ of the pocket boundary formed by the left sides of flow lines with angles~$\theta_1$ and~$\theta_2$, respectively, and the images~$R_n^{\theta_1}$ and~$R_n^{\theta_2}$ of the pocket boundary formed by the right sides of flow lines with angles~$\theta_1$ and~$\theta_2$, respectively.  Note that $\wt{h} = \lim_n \wt{h}_n$.  Consequently, the boundary data for~$\wt{h}$ takes the same form.  Let $L^{\theta_1}$, $L^{\theta_2}$, $R^{\theta_1}$, and $R^{\theta_2}$ be the four marked boundary segments for the boundary data of $\wt{h}$.  If $L^{\theta_1} \neq \emptyset$ or $R^{\theta_2} \neq \emptyset$, then $\diam(L_n^{\theta_1})$ or $\diam(L_n^{\theta_2})$ is bounded from below for arbitrarily large values of~$n$.  This is a contradiction because it is easy to see that on this event, the conformal radius of~$P_{n+1}(z)$ as seen from~$z$ decreases by a uniformly positive amount with uniformly positive probability.  Consequently, $L^{\theta_1} = \emptyset$ and $R^{\theta_2} = \emptyset$ almost surely.  That is, the boundary data for~$\wt{h}$ is in fact as illustrated in the left side of Figure~\ref{fig::lightcone_pocket}, as desired.
\end{proof}

\begin{figure}[ht!]
\begin{center}
\subfigure[]{
\includegraphics[scale=0.85, page=1]{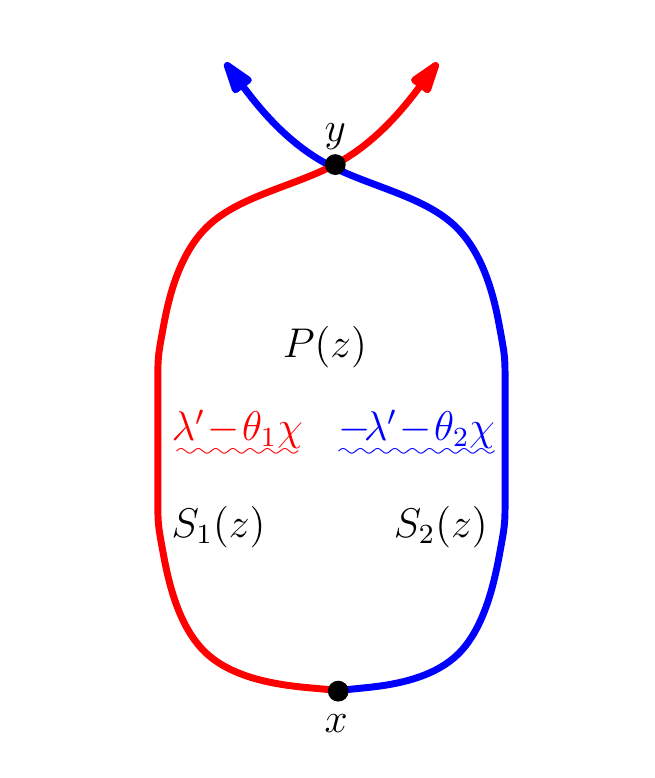}}
\hspace{0.1\textwidth}
\subfigure[]{
\includegraphics[scale=0.85, page=2]{figures/lightcone_pocket}}
\end{center}
\vspace{-0.03\textheight}
\caption{\label{fig::lightcone_pocket} Shown on the left is a pocket~$\pocket{z}$ of $\lightcone(\theta_1,\theta_2)$ containing a given point~$z$ and the boundary data for the conditional law of~$h$ given $\lightcone(\theta_1,\theta_2)$ on $\partial \pocket{z}$.  Note that it is not possible to draw~$\theta_1$-angle (resp.~$\theta_2$-angle) flow lines of~$h$ contained in~$\pocket{z}$ which start from points on~$\side{2}{z}$ (resp.~$\side{1}{z}$).  On the right side, the extra~$\theta_2$-angle flow lines have been drawn in blue to indicate how the paths are ordered using an~$\SLE_{\kappa'}$ counterflow line~$\eta'$.  The dark green path indicates the part~$\eta'$ that fills the right side of~$\side{2}{z}$, the orange path indicates the part of~$\eta'$ which travels from the opening point~$\openL$ to the closing point~$\closeL$ of~$\pocket{z}$, and the light green path indicates the part of~$\eta'$ after it has hit~$\closeL$.  The colored arrows indicate the direction in which the different segments of~$\eta'$ are traveling.  In particular,~$\eta'$ fills the right side of~$\side{2}{z}$ before entering (the interior of) $\pocket{z}$.  Since it has to hit the points on~$\side{1}{z}$ in the reverse order in which they are drawn by~$\sideflow{1}{z}$, after reaching $\openL$, $\eta'$ enters into the interior of $\pocket{z}$ and then travels to $\closeL$.  As it travels up to $\closeL$, it visits point on the left side of $\side{2}{z}$, does not hit $\side{1}{z}$, and does not leave $\ol{\pocket{z}}$.  After reaching $\closeL$, it then visits the points of $\side{1}{z}$ in the reverse order in which they drawn by $\sideflow{1}{z}$.  While it does so, it makes excursions both into and outside of $\pocket{z}$.}
\end{figure}

Throughout, we shall refer to the point~$\openL$ in the statement of Lemma~\ref{lem::form_pockets} as the {\bf opening point} of~$\pocket{z}$.  If we want to emphasize the dependency of~$\openL$ on~$z$, we will write~$\open{z}$ for~$\openL$.  For a generic pocket~$P$, we will write~$\open{P}$ for the opening point of~$P$.  Similarly, we will refer to the point~$\closeL$ in the statement of Lemma~\ref{lem::form_pockets} as the {\bf closing point} of~$\pocket{z}$.  As before, we will write~$\close{z}$ if we want to emphasize the dependency on~$z$ and write~$\close{P}$ for the closing point of a generic pocket~$P$.  We will also use the notation~$\side{j}{z}$ introduced in the statement of Lemma~\ref{lem::form_pockets} to indicate the~$\theta_j$-angle side of~$\partial P(z)$ for $j=1,2$ and write~$\side{j}{P}$ to indicate the same for a generic pocket~$P$.  If~$P$ or~$z$ is understood from the context, then we will simply write~$\side{j}$ for $j=1,2$.  Finally, we note that~$\side{j}{z}$ is equal to the flow line~$\sideflow{j}{z}$ of~$h$ with angle~$\theta_j$ starting from~$\open{z}$ and stopped upon hitting~$\close{z}$.  We will write~$\sideflow{j}{P}$ to indicate these flow lines for a generic pocket~$P$ and write~$\sideflow{j}$ if either~$P$ or~$z$ is understood from the context.  We will now use Lemma~\ref{lem::form_pockets} to show that the definition of the light cone introduced in this section agrees with the one given in the introduction.

\begin{lemma}
\label{lem::lightcone_approximation}
Fix $\theta_1 \leq \theta_2$ with $\theta_2-\theta_1 \leq \pi$.  Let $\lightcone(\theta_1,\theta_2)$ be as defined in the beginning of the subsection and let $\wh{\lightcone}(\theta_1,\theta_2)$ be the closure of the set of points accessible by angle-varying trajectories of $h$ starting from $-i$ with angles which are rational and contained in $[\theta_1,\theta_2]$ or equal to $\theta_1$ or $\theta_2$ and which change angles at most a finite number of times and only at positive rational times.  (This is the definition of the light cone given in the introduction.)  Then $\lightcone(\theta_1,\theta_2) = \wh{\lightcone}(\theta_1,\theta_2)$ almost surely. 
\end{lemma}
\begin{proof}
We may assume without loss of generality that $\theta_1 < \theta_2$ since if $\theta_1=\theta_2$ then the result is trivially true because both $\lightcone(\theta_1,\theta_2)$ and $\wh{\lightcone}(\theta_1,\theta_2)$ are equal to the flow line of $h$ starting from $-i$ with angle $\theta_1=\theta_2$.  It is clear from the definition that $\lightcone(\theta_1,\theta_2) \subseteq \wh{\lightcone}(\theta_1,\theta_2)$ almost surely, so we just need to prove the reverse inclusion.  We first suppose that $\theta_2 - \theta_1 < \theta_c$.  In this case, the result follows because, for each fixed $z \in \D$, the flow line interaction rules \cite[Theorem~1.7]{MS_IMAG4} and Lemma~\ref{lem::form_pockets} imply that an angle-varying trajectory with angles which are rational and contained in $[\theta_1,\theta_2]$ or equal to $\theta_1$ or $\theta_2$ which changes angles at most a finite number of times cannot enter the pocket $\pocket{z}$ of $\lightcone(\theta_1,\theta_2)$ which contains $z$.  Indeed, a flow line of angle $\theta_2$ cannot cross a flow line of angle $\theta_1$ from left to right since $\theta_2 > \theta_1$ and likewise a flow line of angle $\theta_1$ cannot cross a flow line of angle $\theta_2$ from right to left.  The case that $\theta_2 - \theta_1 \geq \theta_c$ follows since for these values we know that both $\lightcone(\theta_1,\theta_2)$ and $\wh{\lightcone}(\theta_1,\theta_2)$ are equal to the set of points which lie between their left and right boundaries.
\end{proof}

Fix angles $\theta_1 < \theta_2$ with $\theta_2 - \theta_1 < \theta_c$ and $\theta_2-\theta_1 \leq \pi$.  Assume that the boundary data of~$h$ is such that the flow lines $\eta_1,\eta_2$ starting from~$-i$ with angles $\theta_1,\theta_2$ almost surely do not hit the continuation threshold (as defined in just before the statement of \cite[Theorem~1.1]{MS_IMAG}).  That is, they both connect~$-i$ to~$i$.  Let~$\eta'$ be the counterflow line of $h+(\theta_2-\tfrac{\pi}{2})\chi$ starting from~$i$.  Then the left boundary of~$\eta'$ stopped upon hitting a point $z \in \D$ is equal to the flow line starting from~$z$ with angle~$\theta_2$.  We are now going to use the flow line interaction rules \cite[Theorem~1.7]{MS_IMAG4} to explain how~$\eta'$ interacts with a pocket~$\pocket{z}$ of~$\lightcone(\theta_1,\theta_2)$.  See Figure~\ref{fig::lightcone_pocket} for an illustration.  If we start a flow line~$\eta_w$ with angle~$\theta_2$ from a point~$w$ inside of~$\pocket{z}$, then it has to merge with~$\sideflow{2}{z}$ on its left side.  Indeed, this is obviously true for topological reasons if~$\eta_w$ merges with~$\sideflow{2}{z}$ before leaving~$\pocket{z}$.  If~$\eta_w$ first leaves~$\pocket{z}$ before merging into~$\sideflow{2}{z}$, then it necessarily crosses~$\sideflow{1}{z}$ from the right to the left.  If~$\eta_w$ were to subsequently wrap around and merge with~$\sideflow{2}{z}$ on its right side, then it would be forced to cross~$\sideflow{1}{z}$ a second time, which is a contradiction to \cite[Theorem~1.7]{MS_IMAG4}.  This proves the claim since flow lines with the same angle almost surely merge.  Similarly, if we start a flow line from a point~$w$ on~$\side{1}{z}$ then it merges with~$\sideflow{2}{z}$ on its left side.  Consequently, it follows from \cite[Theorem~1.13]{MS_IMAG4} that:
\begin{enumerate}
\item $\eta'$ enters (the interior of) $\pocket{z}$ at $\open{z}$ after filling the right side of $\side{2}{z}$.
\item Upon entering $\pocket{z}$, $\eta'$ visits points on the left side of $\side{2}{z}$ as it travels from $\open{z}$ to $\close{z}$.  It does not touch $\side{1}{z}$ until hitting $\close{z}$.
\item Upon hitting $\close{z}$, it visits the points of $\side{1}{z}$ in the reverse order in which they are drawn by $\sideflow{1}{z}$ and, while doing so, $\eta'$ makes excursions both into and out of $\pocket{z}$.
\end{enumerate}
We are now going to extract from this and the continuity of space-filling $\SLE_{\kappa'}$ the local finiteness of the pockets of the light cone.

\begin{lemma}
\label{lem::locally_finite}
Suppose that we have the setup described just above (in particular, the boundary data of $h$ is such that the left and right boundaries $\eta_1,\eta_2$ of $\lightcone(\theta_1,\theta_2)$ almost surely do not hit the continuation threshold before hitting $i$).  The pockets of $\lightcone(\theta_1,\theta_2)$ are almost surely locally finite: that is, for each $\epsilon > 0$, the number of pockets of $\lightcone(\theta_1,\theta_2)$ with diameter at least $\epsilon$ is finite almost surely.
\end{lemma}
\begin{proof}
The result trivially holds for $\theta_2 - \theta_1 \geq \theta_c$ because then~$\lightcone(\theta_1,\theta_2)$ is space-filling hence does not have pockets which lie between~$\eta_1$ and~$\eta_2$.  The pockets which are not surrounded by~$\eta_1$ and~$\eta_2$ are locally finite because~$\eta_1$ and~$\eta_2$ are continuous paths.  We now suppose that $\theta_2 - \theta_1 < \theta_c$ so that~$\lightcone(\theta_1,\theta_2)$ has pockets which lie between~$\eta_1$ and~$\eta_2$.  Since the components of $\D \setminus (\eta_1 \cup \eta_2)$ are locally finite, it suffices to show that the pockets of~$\lightcone(\theta_1,\theta_2)$ which are contained in a given component are locally finite.

Fix such a component~$C$ and let~$\eta'$ be the space-filling $\SLE_{\kappa'}$ process starting from~$y$, the last point on~$\partial C$ hit by~$\eta_1$ and~$\eta_2$, and targeted at~$x$, the first point on~$\partial C$ hit by~$\eta_1$ and~$\eta_2$.  We choose~$\eta'$ so that its left boundary stopped upon hitting any given point is equal to the flow line of~$h$ with angle~$\theta_2$ starting from that point.  Then~$\eta'$ interacts with a pocket~$\pocket{z}$ of~$\lightcone(\theta_1,\theta_2)$ for $z \in C$ in the same manner as the counterflow line described before the statement of the lemma except that it completely fills~$\side{2}{z}$ while traveling from~$\open{z}$ to~$\close{z}$.  Note that for disjoint pockets~$\pocket{z}$ and~$\pocket{w}$ of~$\lightcone(\theta_1,\theta_2)$ contained in~$C$, the time-interval~$I_z$ in which~$\eta'$ travels from~$\open{z}$ to~$\close{z}$ is disjoint from the time-interval~$I_w$ in which it travels from~$\open{w}$ to~$\close{w}$.  Moreover, for each $z \in \D$, $\eta'(I_z)$ contains~$\side{2}{z}$.  Consequently, it follows from the continuity of space-filling $\SLE_{\kappa'}$ that the number of pockets~$P$ such that $\diam(\side{2}{P}) \geq \epsilon$ is finite almost surely.  The same is also true for the number of pockets~$P$ such that $\diam(\side{1}{P}) \geq \epsilon$ because we can take a space-filling $\SLE_{\kappa'}$ whose right boundary stopped upon hitting a given point~$z$ is given by the flow line starting from~$z$ with angle~$\theta_1$ in place of~$\eta'$ and then apply the same analysis.  This completes the proof since the triangle inequality implies that $\diam(P) \leq \diam(\side{1}{P}) + \diam(\side{2}{P})$ for any pocket~$P$.
\end{proof}

\begin{figure}[ht!]
\begin{center}
\subfigure[]{
\includegraphics[scale=0.85, page=1]{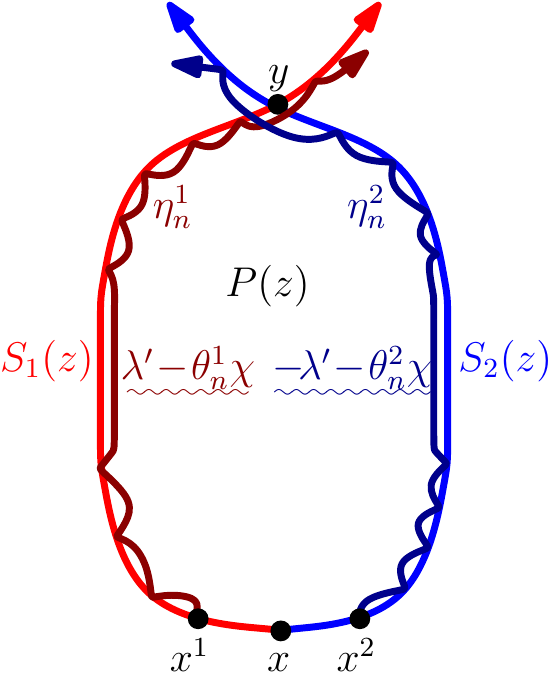}}
\hspace{0.15\textwidth}
\subfigure[]{
\includegraphics[scale=0.85, page=2]{figures/lightcone_continuous}}
\end{center}
\caption{\label{fig::lightcone_continuous} Shown on the left side is the pocket $\pocket{z}$ of $\lightcone(\theta_1,\theta_2)$ containing~$z$.  Its opening (resp.\ closing) point is $\openL$ (resp.\ $\closeL$).  Suppose that $(\theta_n^1)$ (resp.\ $(\theta_n^2)$) is a sequence of angles which increase to~$\theta_1$ (resp.\ decrease to $\theta_2$).  We take~$\eta_n^1$ (resp.\ $\eta_n^2$) to be a flow line of angle $\theta_n^1$ (resp.\ $\theta_n^2$) starting from $x^1 \in \partial \pocket{z}$ (resp.\ $x^2 \in \partial \pocket{z}$).  As $n \to \infty$, $\eta_n^1$ and~$\eta_n^2$ converge in the Hausdorff topology to the segments of~$\side{1}{z}$ and~$\side{2}{z}$, respectively, which connect~$x^1$ and~$x^2$ to~$y$.  The right is the same as the left except we have drawn dual paths $\wt{\eta}_n^1,\wt{\eta}_n^2$ starting from points on $\eta_n^1,\eta_n^2$, respectively.  Explicitly, $\eta_n^1$ (resp.\ $\eta_n^2$) has angle $\theta_n^1-\pi$ (resp.\ $\theta_n^2+\pi$).  These paths will intersect and bounce off each other as shown.  By the flow line interaction rules, $\wt{\eta}_n^1$ cannot cross either~$\eta^2$ or~$\eta_n^2$ but can cross out of~$\pocket{z}$ through the clockwise segment of~$\side{2}{z}$ from~$x$ to~$x^1$ and the symmetric fact holds for $\wt{\eta}_n^2$.  Since an angle varying flow line with angles contained in $[\theta_n^1,\theta_n^2]$ cannot cross from the right to the left (resp.\ left to the right) of $\wt{\eta}_n^1$ (resp.\ $\wt{\eta}_n^2$), it follows that the pocket of $\lightcone(\theta_n^1,\theta_n^2)$ which contains $z$ almost surely contains the light blue region on the right.  This allows us to prove the continuity of the law of $\lightcone(\theta_1,\theta_2)$ in $\theta_1,\theta_2$ with respect to the Hausdorff topology because the Hausdorff distance between~$\pocket{z}$ and the blue region will with probability tending to~$1$ decrease to $0$ as we take a limit first as $n \to \infty$ and then as $x^1,x^2 \to x$, and then finally the starting points of $\wt{\eta}_n^1,\wt{\eta}_n^2$ to $\openL$ as well.}
\end{figure}

We are now going to establish the continuity of the law of $\lightcone(\theta_1,\theta_2)$ in $\theta_1 \leq \theta_2$ with $\theta_2-\theta_1 \leq \pi$ with respect to the Hausdorff topology.  See Figure~\ref{fig::lightcone_continuous} for an illustration of the setup and the proof.

\begin{proposition}
\label{prop::lightcone_continuous}
Suppose that we have the same setup as in Lemma~\ref{lem::locally_finite} and that $\theta_1 \leq \theta_2$ are angles with $\theta_2 - \theta_1 \leq \pi$.  Let $(\theta_n^1)$, $(\theta_n^2)$ be sequences of angles with $\theta_n^1 \leq \theta_n^2$ and $\theta_n^2 -\theta_n^1 \leq \pi$ for all $n \in \N$ such that $\theta_n^j \to \theta_j$ as $n \to \infty$ for $j=1,2$.  Then $\lightcone(\theta_n^1,\theta_n^2) \to \lightcone(\theta_1,\theta_2)$ as $n \to \infty$ almost surely with respect to the Hausdorff topology.
\end{proposition}

\begin{remark}
\label{rem::lightcone_not_continuous}
Proposition~\ref{prop::lightcone_continuous} implies that for a \emph{fixed} choice of $\theta_1 \leq \theta_2$, we have that $\lightcone(\theta_n^1,\theta_n^2) \to \lightcone(\theta_1,\theta_2)$ almost surely.  It does not imply that $(\theta_1,\theta_2) \mapsto \lightcone(\theta_1,\theta_2)$ is a continuous function with respect to the Hausdorff topology for a fixed realization of $h$.  Indeed, this statement is not true because the left boundary of $\lightcone(\theta_1,\theta_2)$ is the flow line of $h$ with angle $\theta_1$ and for a fixed realization of $h$ the map which takes an angle to the flow line of $h$ starting from that angle is not continuous with respect to the Hausdorff topology.  Indeed, if this were true then the \emph{fan} defined in \cite{MS_IMAG} would almost surely have positive Lebesgue measure but it is shown in \cite{MS_IMAG} that the Lebesgue measure is zero almost surely.  In fact, it is shown in \cite{LIGHTCONE_DIMENSION} that the dimension of the fan is the same as the dimension of a single $\SLE_\kappa$ path.
\end{remark}

\begin{proof}[Proof of Proposition~\ref{prop::lightcone_continuous}]
We are going to give the proof in the case that $(\theta_n^1)$ increases to~$\theta_1$ and $(\theta_n^2)$ decreases to~$\theta_2$.  We will also assume that $\theta_2 - \theta_1 < \theta_c$.  The proof in the other possible cases is similar.  By Lemma~\ref{lem::locally_finite}, we know that the pockets of~$\lightcone(\theta_1,\theta_2)$ are locally finite.  Fix $\epsilon > 0$ and let $P_1,\ldots,P_n$ be the pockets of~$\lightcone(\theta_1,\theta_2)$ which have diameter at least~$\epsilon$.  For each~$j$, we let $\openL_j = \open{P_j}$ (resp.\ $\closeL_j = \close{P_j}$) be the opening (resp.\ closing) point of~$P_j$.  Fix $\delta \in (0,\epsilon)$ and, for each~$j$, let~$x_j^1$ (resp.\ $x_j^2$) be a point on~$\side{1}{P_j}$ (resp.\ $\side{2}{P_j}$) with distance at most~$\delta$ from~$x_j$.  Let~$\eta_{j,n}^1$ (resp.\ $\eta_{j,n}^2$) be the flow line of~$h$ starting from~$x_j^1$ (resp.\ $x_j^2$) with angle~$\theta_n^1$ (resp.\ $\theta_n^2$). As $n \to \infty$, these paths stopped upon exiting~$\ol{P}_j$ almost surely converge in the Hausdorff topology to the segments of~$\side{1}{P_j}$ and~$\side{2}{P_j}$ which start from~$x_j^1$ and~$x_j^2$, respectively, and terminate at~$\closeL_j$.  Indeed, this follows for~$\eta_{j,n}^1$ because it is an $\SLE_\kappa(\rho^L;\rho_1^R,\rho_2^R)$ process in~$P_j$ with force points located at $(x_j^1)^-$, $(x_j^1)^+$, and~$x_j^2$ and $\rho^L \downarrow -2$ as $n \to \infty$.  This follows for~$\eta_{j,n}^2$ for an analogous reason.

Fix $\delta_1 \in (\delta,\epsilon)$; we will shortly send $\delta \downarrow 0$ while leaving~$\delta_1$ and~$\epsilon$ fixed.  For each~$n$, let~$x_{n,j}^1$ be a point on~$\eta_{n,j}^1$ which has distance~$\delta_1$ from~$\openL_j$ and let~$\wt{\eta}_{n,j}^1$ be the flow line of~$h$ starting from~$x_{n,j}^1$ with angle $\theta_n^1-\pi$ (the angle dual to that of~$\eta_{n,j}^1$).  We define~$x_{n,j}^2$ and~$\wt{\eta}_{n,j}^2$ similarly (the angle of~$\wt{\eta}_{n,j}^2$ is $\theta_n^2+\pi$).  Let~$C_j$ be the component of $P_j \setminus (\eta_{n,j}^1 \cup \eta_{n,j}^2)$ which $\wt{\eta}_{n,j}^1,\wt{\eta}_{n,j}^2$ enter immediately upon getting started (there exists such a component with probability tending to~$1$ as $n \to \infty$).  Then the joint law of~$\wt{\eta}_{n,j}^1, \wt{\eta}_{n,j}^2$ in~$C_j$ stopped upon hitting $B(\openL_j,2\delta)$ is absolutely continuous with respect to that of the pair of paths~$(\wh{\eta}_{n,j}^1,\wh{\eta}_{n,j}^2)$ which are distributed as in the case that the boundary data along~$C_j$ takes the same form as if it were a pocket of a light cone with angle range $[\theta_n^1,\theta_n^2]$ and with opening and closing points~$\openL_j$ and~$\closeL_j$, respectively.  Moreover, by \cite[Lemma~2.1]{LIGHTCONE_DIMENSION}, the Radon-Nikodym derivative is bounded from above and below by universal finite and positive constants which do not depend on~$n$.  By the flow line interaction rules \cite[Theorem~1.5]{MS_IMAG}, $\wh{\eta}_{n,j}^1,\wh{\eta}_{n,j}^2$ almost surely intersect before exiting~$C_j$.  Consequently, sending first $n \to \infty$, then $\delta \downarrow 0$, we see that the probability that $\wt{\eta}_{n,j}^1$ intersects $\wt{\eta}_{n,j}^2$ before hitting $B(\openL_j,\delta)$ tends to~$1$.  Moreover, the diameter of the paths up until intersecting almost surely tends to zero upon taking another limit as $\delta_1 \downarrow 0$.  The desired result follows because the pocket of~$\lightcone(\theta_n^1,\theta_n^2)$ which contains~$z$ is contained in~$\pocket{z}$ and contains the component of $C_j \setminus (\wt{\eta}_{n,j}^1 \cup \wt{\eta}_{n,j}^2)$ containing~$z$ on the event that~$\wt{\eta}_{n,j}^1$ and~$\wt{\eta}_{n,j}^2$ intersect before leaving $B(x_j,2\delta)$ provided~$\delta$ is small enough.  See the caption of Figure~\ref{fig::lightcone_continuous} for further explanation of this final point.
\end{proof}

\begin{proposition}
\label{prop::pocket_boundaries_continuous}
Suppose that we have the same setup as in Lemma~\ref{lem::locally_finite} and that $\theta_1 \leq \theta_2$ are angles with $\theta_2 - \theta_1 < \theta_c$ and $\theta_2 - \theta_1 \leq \pi$.  Let $(\theta_n^1)$, $(\theta_n^2)$ be sequences of angles with $\theta_n^1 \leq \theta_n^2$ and $\theta_n^2 -\theta_n^1 \leq \pi$ for all $n \in \N$ such that $\theta_n^j \to \theta_j$ as $n \to \infty$ for $j=1,2$.  For each $z \in \D$ and $n \in \N$, let $\eta_j^n(z)$ be the flow line which forms the $\theta_j^n$-angle boundary of the pocket of $\lightcone(\theta_n^1,\theta_n^2)$ which contains $z$.  Then $\eta_j^n(z) \to \sideflow{j}{z}$ for $j=1,2$ almost surely as $n \to \infty$ with respect to the uniform topology modulo parameterization.
\end{proposition}
\begin{proof}
This follows from the same argument used to prove Proposition~\ref{prop::lightcone_continuous}.
\end{proof}

\subsection{Explorations and continuity}
\label{subsec::explorations}

\begin{figure}[ht!]
\begin{center}
\includegraphics[scale=0.85,page=3]{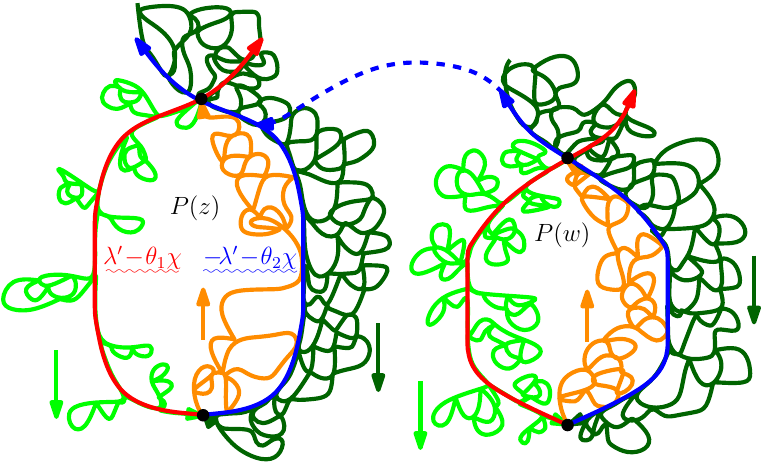}
\end{center}
\caption{\label{fig::lightcone_pocket2} (Continuation of Figure~\ref{fig::lightcone_pocket}.)  Shown is a second pocket $\pocket{w}$ of $\lightcone(\theta_1,\theta_2)$.  If the $\theta_2$-angle flow line $\sideflow{2}{w}$ which generates $\side{2}{w}$ merges into the right side of the $\theta_2$-angle flow line $\sideflow{2}{z}$ which generates $\side{2}{z}$ of $\partial \pocket{z}$ as illustrated, then the counterflow line $\eta'$ visits (the interior of) $\pocket{w}$ before visiting (the interior of) $\pocket{z}$.  This determines (and is the same as) the order in which the trajectory we consider which explores $\lightcone(\theta_1,\theta_2)$ visits $\pocket{z}$ and $\pocket{w}$.  The same color scheme for the segments of $\eta'$ as it visits the points of $\partial \pocket{w}$ is used as in Figure~\ref{fig::lightcone_pocket}.
}
\end{figure}

We assume that $\theta_1 < \theta_2$ are angles with $\theta_2 - \theta_1 < \theta_c$ and $\theta_2 - \theta_1 \leq \pi$.  We also assume that the boundary data of~$h$ is such that the flow lines $\eta_1,\eta_2$ with angles $\theta_1,\theta_2$ starting from~$-i$, respectively, almost surely reach~$i$ before hitting the continuation threshold.  Let~$\eta'$ be the counterflow line of $h + (\theta_2-\tfrac{\pi}{2})\chi$ starting from~$i$ and targeted at~$-i$.  By \cite[Theorem~1.4]{MS_IMAG}, the left boundary of~$\eta'$ stopped upon hitting any point is equal to the flow line of angle~$\theta_2$ starting from that point.  We will use the path~$\eta'$ to order the points on~$\lightcone(\theta)$ then use the continuity of~$\eta'$ to show that there exists a continuous, non-crossing path whose range is equal to~$\lightcone(\theta_1,\theta_2)$ and which visits the points of~$\lightcone(\theta_1,\theta_2)$ in this order.  We will then show that the path has a continuous chordal Loewner driving function and, in certain special cases, yields a local set for~$h$ when drawn up to any stopping time.  In the next section, we will use these facts to complete the proof of Theorem~\ref{thm::coupling} by showing that the corresponding path (in a slightly modified setup) evolves as the appropriate~$\SLE_\kappa(\rho)$ process and is coupled with and determined by the field in the desired manner.  This will also give Theorem~\ref{thm::continuous}.  The path which traverses~$\lightcone(\theta_1,\theta_2)$ is constructed in the following manner.
\begin{enumerate}
\item Suppose that $z,w \in \D$ are distinct.  We say that~$\pocket{z}$ comes before~$\pocket{w}$ if~$\eta'$ visits~$\open{z}$ before~$\open{w}$.  Equivalently, $\pocket{z}$ comes before~$\pocket{w}$ if the flow line of~$h$ starting from~$\open{z}$ with angle~$\theta_2$ merges with the flow line of angle~$\theta_2$ starting from~$\open{w}$ on its right side.
\item We take~$\eta$ to be the concatenation of the paths~$\sideflow{1}{z}$ using the same ordering as for the pockets~$\pocket{z}$.
\end{enumerate}

We will now use the continuity of~$\eta'$ to deduce the continuity of~$\eta$.

\begin{lemma}
\label{lem::continuous}
The trajectory~$\eta$ from~$i$ to~$-i$ in~$\ol{\D}$ described above is almost surely continuous.
\end{lemma}
\begin{proof}
Let $\eta_1, \eta_2$ be the flow lines of~$h$ starting from~$-i$ with angles $\theta_1,\theta_2$, respectively, as before, and let~$\eta'$ be the counterflow line of $h+(\theta_2-\tfrac{\pi}{2})\chi$ starting from~$i$ and targeted at~$-i$.  From \cite[Theorem~1.3]{MS_IMAG}, we know that~$\eta'$ is almost surely continuous.  We are going to prove the continuity of~$\eta$ in two steps.  First, we will construct an intermediate path by starting with~$\eta'$ and then excising the excursions that it makes into~$\lightcone(\theta_1,\theta_2)$.  Second, we will modify this intermediate path to get~$\eta$.

Let $I = [0,\infty) \setminus (\eta')^{-1}(\lightcone(\theta_1,\theta_2))$.  Since~$\eta'$ is continuous, $I \subseteq [0,\infty)$ is open, hence we can write $I = \cup_j (s_j,t_j)$ as a countable, disjoint union of open intervals.  Note that for each~$j$ there exists $z \in \D$ such that $\eta'((s_j,t_j)) \subseteq \pocket{z}$.  Suppose that $\eta'(s_j) \in \side{1}{z}$.  Since~$\side{1}{z}$ is contained in the range of~$\eta'$ and~$\eta'$ visits of the points of~$\side{1}{z}$ in the reverse chronological order in which they are drawn by~$\sideflow{1}{z}$, it must be that $\eta'(s_j)=\eta'(t_j)$.  Consequently, letting $\wt{\eta}|_{[0,\infty) \setminus I} = \eta'|_{[0,\infty) \setminus I}$ and $\wt{\eta}|_{(s_j,t_j)} = \eta'(s_j) = \eta'(t_j)$ for each $j \in \N$ such that $\eta'(s_j) \in \side{1}{z}$, we see that~$\wt{\eta}$ is almost surely continuous.  Note that after filling the right side of~$\side{2}{z}$ for a pocket~$\pocket{z}$ and then after hitting~$\open{z}$ for the first time,~$\wt{\eta}$ travels inside~$\pocket{z}$ starting from~$\open{z}$ until reaching~$\close{z}$ while bouncing off the left side of~$\side{2}{z}$ and does not hit the right side of~$\side{1}{z}$.  The amount of time that this takes is equal to the amount of time it takes~$\eta'$ to travel from~$\open{z}$ to~$\close{z}$.  Next, $\wt{\eta}$ fills~$\side{1}{z}$ until reaching~$\open{z}$.  While filling~$\side{1}{z}$, it makes excursions out of~$\pocket{z}$ but never into (the interior of) $\pocket{z}$.

Recall from Lemma~\ref{lem::locally_finite} that the pockets of~$\lightcone(\theta_1,\theta_2)$ are almost surely locally finite.  Let $(P_n)$ be an ordering of the pockets of~$\lightcone(\theta_1,\theta_2)$ such that $\diam(P_n) \geq \diam(P_{n+1})$ for all~$n$.  (For example, we can order the the pockets by diameter and then break ties using a fixed ordering of the rationals.)  For each~$j$, we let~$\wt{\eta}_j$ be the path which agrees with~$\wt{\eta}$ in~$P_m$ for $m \geq j+1$ and, for $1 \leq m \leq j$, follows~$\sideflow{1}{P_m}$ rather than~$\wt{\eta}$ while traveling from~$\open{z}$ to~$\close{z}$ (but in the same interval of time).  The local finiteness of the $(P_j)$ implies that the sequence $(\wt{\eta}_j)$ is Cauchy with respect to the uniform topology.  Therefore the sequence $(\wt{\eta}_j)$ has a continuous limit~$\wh{\eta}$.

To complete the proof, we are going to argue that~$\wh{\eta}$ is the same as~$\eta$.  We begin by reparameterizing~$\wh{\eta}$ by excising those intervals of time which correspond to the excursions that~$\eta'$ makes into pockets of~$\lightcone(\theta_1,\theta_2)$ starting from~$\side{1}{P}$ for a pocket~$P$.  We do not change the time in which~$\wh{\eta}$ is drawing the boundaries~$\side{1}{P}$ themselves.  By the continuity of~$\eta'$, it is easy to see that this reparameterization is continuous (the set of these excursions is locally finite).  Moreover, the set of times that~$\wh{\eta}$ is drawing the boundaries~$\side{1}{P}$ has full Lebesgue measure and, in particular, is dense.  This proves that it~$\eta$ can be reparameterized so that it extends continuously off the intervals of time in which it is drawing the~$\theta_1$-angle boundaries, which proves the desired result.
\end{proof}

\begin{lemma}
\label{lem::continuous_loewner}
The path~$\eta$ from Lemma~\ref{lem::continuous} has a continuous chordal Loewner driving function.
\end{lemma}
\begin{proof}
We will prove the result using \cite[Proposition~6.12]{MS_IMAG}.  We first apply a conformal change of coordinates $\D \to \h$ which sends~$i$ to~$0$ and~$-i$ to~$\infty$ so that we may assume without loss of generality that we are working on~$\h$.  That the first criterion from \cite[Proposition~6.12]{MS_IMAG} is satisfied by~$\eta$ follows from Lemma~\ref{lem::continuous} and the way that we have constructed~$\eta$ from~$\eta'$.  We will now check the second criterion.  That is,~$\eta$ almost surely does not trace itself or~$\partial \h$.  If we parameterize~$\eta$ as in the end of the proof of Lemma~\ref{lem::continuous}, then we know that it spends Lebesgue almost all of its time drawing the~$\theta_1$-angle boundaries of the pockets of~$\lightcone(\theta_1,\theta_2)$.  When drawing such a boundary,~$\eta$ does not hit the past of its range except at the opening and closing points of the corresponding pockets.  Moreover, it also cannot trace the domain boundaries in these intervals.  Consequently, the claimed result follows.
\end{proof}

We are next going to argue that the path~$\eta$ together with the left and right boundaries~$\eta_2$ and~$\eta_1$, respectively, of~$\lightcone(\theta_1,\theta_2)$ is local (in the sense of \cite{SchrammShe10}) for and almost surely determined by~$h$.

\begin{proposition}
\label{prop::ordering_local}
For each $t \geq 0$, let~$\CF_t$ be the $\sigma$-algebra generated by~$\eta|_{[0,t]}$ and the left and right boundaries~$\eta_2$ and~$\eta_1$, respectively, of~$\lightcone(\theta_1,\theta_2)$.  For each $(\CF_t)$-stopping time~$\tau$, $\eta([0,\tau]) \cup \eta_1 \cup \eta_2$ is a local set for and almost surely determined by~$h$.
\end{proposition}

Let~$\eta'$ be the counterflow line of $h+(\theta_2-\tfrac{\pi}{2})\chi$ starting from~$i$ and targeted at~$-i$.  Then the left boundary of~$\eta'$ stopped upon hitting a point~$z$ is equal to the flow line starting from~$z$ with angle~$\theta_2$.  To prove Proposition~\ref{prop::ordering_local}, we are going to describe a ``local'' construction of~$\eta$ from~$\eta'$ (one which will only require us first to observe the left and right boundaries~$\eta_1$ and~$\eta_2$, respectively, of~$\lightcone(\theta_1,\theta_2)$ but not all of~$\lightcone(\theta_1,\theta_2)$).  We begin by using~$\eta'$ to define paths as follows.  Fix~$\epsilon > 0$.  Let~$\tau_{\epsilon,1}$ be the first time~$t$ that there exists a flow line~$\eta_{\epsilon,1}^R$ of~$h$ with angle~$\theta_1$ starting from~$\eta'(t)$ and which crosses into~$\eta'([0,t])$ on its left side (i.e., the part of the outer boundary of~$\eta'([0,t])$ which is described by a flow line of angle~$\theta_2$ starting from~$\eta'(t)$) such that the following is true: the pocket formed by the left side of~$\eta'([0,t])$ and the range of this path drawn up until crossing into the left side of~$\eta'([0,t])$ has diameter at least~$\epsilon$.  (Throughout, we shall write~$\eta_{\epsilon,1}^R$ to mean the path stopped at the time of first hitting the left side of~$\eta'([0,\tau_{\epsilon,1}])$.)  Note that the pocket will have diameter at least~$\epsilon$ if either:
\begin{enumerate}
\item $\eta_{\epsilon,1}^R$ has diameter at least~$\epsilon$ or
\item $\eta_{\epsilon,1}^R$ has diameter less than~$\epsilon$ hence closes the pocket before leaving the $\epsilon$-neighborhood of $\eta'([0,\tau_{\epsilon,1}])$.
\end{enumerate}
In particular, each of the two possibilities can be determined by observing the values of~$h$ in an $\epsilon$-neighborhood of $\eta'([0,\tau_{\epsilon,1}])$.

We then let $\eta_{\epsilon,1}'$ be the path which agrees with $\eta'$ until time $\tau_{\epsilon,1}$ and then follows $\eta_{\epsilon,1}^R$ until hitting the left side of $\eta'([0,\tau_{\epsilon,1}])$.  Let $P_{\epsilon,1}$ be the pocket thus formed by $\eta_{\epsilon,1}^R$ and the left side of $\eta'([0,\tau_{\epsilon,1}])$.  Note that $\partial P_{\epsilon,1}$ consists of the right side $\eta_{\epsilon,1}^R$ and the left side of a flow line starting from $\eta'(\tau_{\epsilon,1})$ with angle $\theta_2$.  In other words, $\partial P_{\epsilon,1}$ has the same structure as a pocket of $\lightcone(\theta_1,\theta_2)$; recall Lemma~\ref{lem::form_pockets}.  We let $x_{\epsilon,1} = \eta'(\tau_{\epsilon,1})$ be the opening point of $P_{\epsilon,1}$ and let $y_{\epsilon,1}$ be the closing point of $P_{\epsilon,1}$.  Explicitly, $y_{\epsilon,1}$ is the point at which $\eta_{\epsilon,1}^R$ crosses into $\eta'([0,\tau_{\epsilon,1}])$.  Moreover, $\eta'|_{[\tau_{\epsilon,1},\infty)}$ interacts with $P_{\epsilon,1}$ in the same manner that $\eta'$ interacts with a pocket of $\lightcone(\theta_1,\theta_2)$ as described in Figure~\ref{fig::lightcone_pocket} and Figure~\ref{fig::lightcone_pocket2}.  In particular, $\eta'|_{[\tau_{\epsilon,1},\infty)}$ enters (the interior of) $P_{\epsilon,1}$ at $x_{\epsilon,1}$ and does not leave $\ol{P}_{\epsilon,1}$ or hit $\eta_{\epsilon,1}^R$ until hitting $y_{\epsilon,1}$ for the first time, say at time $\sigma_{\epsilon,1}$.  After hitting $y_{\epsilon,1}$ it visits the points on $\eta_{\epsilon,1}^R$ in the reverse order in which they are drawn by $\eta_{\epsilon,1}^R$.  In particular, $\eta'|_{[\sigma_{\epsilon,1},\infty)}$ makes excursions both into and out of $P_{\epsilon,1}$ and each such excursion starts and ends at the same point on $\eta_{\epsilon,1}^R$ (different excursions, however, are rooted at different points on $\eta_{\epsilon,1}^R$).  We take the part of $\eta_{\epsilon,1}'$ after it has finished drawing $\eta_{\epsilon,1}^R$ to be given by $\eta'|_{[\sigma_{\epsilon,1},\infty)}$ with those excursions of $\eta'$ from $\eta_{\epsilon,1}^R$ into $P_{\epsilon,1}$ excised (we leave the excursions out of $P_{\epsilon,1}$ alone).

Suppose that $k \geq 1$ and that paths $\eta_{\epsilon,1}'$, $\eta_{\epsilon,1}^R$, $\ldots$,$\eta_{\epsilon,k}'$, $\eta_{\epsilon,k}^R$, stopping times $\tau_{\epsilon,1}$, $\sigma_{\epsilon,1}$, $\ldots$, $\tau_{\epsilon,k}$, $\sigma_{\epsilon,k}$, and pockets $P_{\epsilon,1},\ldots,P_{\epsilon,k}$ with opening and closing points $x_{\epsilon,1},y_{\epsilon,1},\ldots,x_{\epsilon,k},y_{\epsilon,k}$ have been defined.  We then let $\tau_{\epsilon,k+1}$ be the first time $t$ after time $\sigma_{\epsilon,k}$ that there is a flow line $\eta_{\epsilon,k+1}^R$ of $h$ with angle $\theta_1$ starting from $\eta_{\epsilon,k}'(t)$ which crosses into the left side of $\eta_{\epsilon,k}'([0,t])$ such that the pocket thus formed has diameter at least $\epsilon$.  We then take $\eta_{\epsilon,k+1}'$ to be the path constructed from $\eta_{\epsilon,k}'$ in the same manner that we constructed $\eta_{\epsilon,1}'$ from $\eta'$ and let $\sigma_{\epsilon,k+1}$ (resp.\ $P_{\epsilon,k+1}$) be the corresponding stopping time (resp.\ pocket).  Finally, we let $x_{\epsilon,k+1}$ (resp.\ $y_{\epsilon,k+1}$) be the opening (resp.\ closing) point of $P_{\epsilon,k+1}$.

For each $\epsilon > 0$, we let $\CP_\epsilon(\theta_1,\theta_2)$ consist of those pockets of $\lightcone(\theta_1,\theta_2)$ which have diameter at least $\epsilon$; recall from Lemma~\ref{lem::continuous} that $\CP_\epsilon(\theta_1,\theta_2)$ is finite almost surely.  Let $\lightcone_\epsilon^R(\theta_1,\theta_2) = \{\side{1}{P} : P \in \CP_\epsilon(\theta_1,\theta_2)\}$.  Let $J_\epsilon = \sup\{j \geq 1 : \tau_{\epsilon,j} < \infty\}$ and let $\CR_\epsilon = \{ \eta_{\epsilon,j}^R : 1 \leq j \leq J_\epsilon\}$ consist of the $\theta_1$-angle boundary segments of the pockets $P_{\epsilon,j}$ (we will explain below that $J_\epsilon < \infty$ almost surely).

We are now going to collect several observations about the exploration procedure that we have just defined.

\begin{lemma}
\label{lem::epsilon_path_properties}
Fix $\epsilon > 0$.  The following are true.
\begin{enumerate}[(i)]
\item\label{it::epsilon_neighborhood_local} Suppose that $\zeta$ is a stopping time for $\eta_{\epsilon,j}'$.  Then the $\epsilon$-neighborhood of $\eta_{\epsilon,j}'([0,\zeta])$ is a local set for $h$.
\item\label{it::path_does_not_enter_pockets} For each $j \geq i$, almost surely $\eta_{\epsilon,j}'$ does not enter (the interior of) $P_{\epsilon,i}$.
\item\label{it::epsilon_pockets_finite} Almost surely, $J_\epsilon < \infty$.
\item\label{it::epsilon_pockets_contained} For each $ 1 \leq j \leq J_\epsilon$ such that $P_{\epsilon,j}$ lies between the left and right boundaries of $\lightcone(\theta_1,\theta_2)$ there almost surely exists $P \in \CP_\epsilon(\theta_1,\theta_2)$ such that $P_{\epsilon,j} \subseteq P$ and $\eta_{\epsilon,j}^R$ emanates from a point on $\side{2}{P}$.
\item\label{it::epsilon_pockets_lightcone} Almost surely, $\lightcone_\epsilon^R(\theta_1,\theta_2)$ is equal to the set which consists of those elements $\eta_{\epsilon,j}^R$ of $\CR_\epsilon$ for $1 \leq j \leq J_\epsilon$ which lie between the left and right boundaries of $\lightcone(\theta_1,\theta_2)$.
\end{enumerate}
\end{lemma}
\begin{proof}
To prove Part~\eqref{it::epsilon_neighborhood_local}, we will use the characterization of local sets given in the first part of \cite[Lemma~3.9]{SchrammShe10}.  We are first going to explain the proof in the case that $j=1$.  Fix $B \subseteq \D$ open and let $\tau_{\epsilon,B}$ be the first time $t$ that $\dist(\eta'(t),B) \leq \epsilon$.  Let $h_B$ the projection of $h$ onto the subspace of functions which are harmonic on $B$.  Then \cite[Theorem~1.2]{MS_IMAG} implies that $\eta'|_{[0,\tau_{\epsilon,B}]}$ is almost surely determined by $h_B$.  Note that the event $\tau_{\epsilon,1} \leq \tau_{\epsilon,B}$ is also almost surely determined by $h_B$ because the set of all flow lines with angle $\theta_1$ starting from points in $\D \setminus B$ and stopped upon exiting $\D \setminus B$ is (simultaneously) almost surely determined by $h_B$.  In particular, we only need to observe these flow lines in an $\epsilon$-neighborhood of $\eta'|_{[0,\tau_{\epsilon,B}]}$ to see if $\tau_{\epsilon,1} \leq \tau_{\epsilon,B}$; recall the discussion after the statement of Proposition~\ref{prop::ordering_local}.  Assume that we are working on the event $\tau_{\epsilon,1} \leq \tau_{\epsilon,B}$.  Then $\eta'|_{[0,\tau_{\epsilon,1}]}$ is almost surely determined by $h_B$ for the same reason.  Let $\tau_{\epsilon,1,B}$ be the first time $t$ that $\dist(\eta_{\epsilon,1}^R(t),B) \leq \epsilon$.  Then $\eta_{\epsilon,1}^R|_{[0,\tau_{\epsilon,1,B}]}$ is also almost surely determined by $h_B$, again for the same reason.  Finally, on the event that $\eta_{\epsilon,1}^R$ terminates in $\eta'([0,\tau_{\epsilon,1}])$ before time $\tau_{\epsilon,1,B}$, it is easy to see that $\eta_{\epsilon,1}'|_{[\sigma_{\epsilon,1},\infty)}$ stopped upon getting within distance $\epsilon$ of $B$ is almost surely determined by $h_B$ because it is given by the counterflow line of $h$ starting from the terminal point of $\eta_{\epsilon,1}^R$ with its excursions into $P_{\epsilon,1}$ excised.  In particular, this is the same as the counterflow line of the conditional GFF $h$ given $\eta'|_{[0,\tau_{\epsilon,1}]}$ and $\eta_{\epsilon,1}^R$ starting from the terminal point of $\eta_{\epsilon,1}^R$.   This proves Part~\eqref{it::epsilon_neighborhood_local} for $j=1$.  The result for $j \geq 2$ follows using a similar argument and induction on $j$.

Part~\eqref{it::path_does_not_enter_pockets} follows because, by our construction, after drawing a pocket we excise all of the excursions that the counterflow line makes into that pocket and the flow line interaction rules imply that a flow line of angle $\theta_1$ (i..e, one of the $\eta_{\epsilon,j}^R$) cannot cross into the interior of such a pocket.

We turn to Part~\eqref{it::epsilon_pockets_finite}.  For each $k \geq 1$, consider the path $\wt{\eta}_{\epsilon,k}'$ which is given by starting with $\eta'$ and then excising the excursions that $\eta'$ makes into the interior of each $P_{\epsilon,j}$ for $1 \leq j \leq k$.  Then each path $\wt{\eta}_{\epsilon,k}'$ is continuous and has the same range as $\eta_{\epsilon,k}'$ by the argument described after Lemma~\ref{lem::lightcone_approximation}.  In particular, the range of $\wt{\eta}_{\epsilon,k}'$ is equal to $\ol{\D} \setminus \cup_{j=1}^k P_{\epsilon,j}$.  As $k \geq 1$ increases, more and more excursions are excised in order to generate $\wt{\eta}_{\epsilon,k}'$.  Thus arguing as in the proof of Lemma~\ref{lem::continuous}, this implies that the limit $\wt{\eta}_\epsilon'$ of $\wt{\eta}_{\epsilon,k}'$ as $k \to \infty$ exists as a uniform limit of continuous paths on a compact interval and $\wt{\eta}_\epsilon'$ is continuous and non-self-crossing.  Moreover, the complement of the range of $\wt{\eta}_\epsilon'$ can only have a finite number of components of diameter larger than $\epsilon > 0$.  Indeed, for otherwise the range of $\wt{\eta}_\epsilon'$ would not be locally connected which in turn would contradict continuity.  This gives Part~\eqref{it::epsilon_pockets_finite}.

We are now going to explain the proof of Part~\eqref{it::epsilon_pockets_contained}.  We first condition $h$ on the $\epsilon$-neighborhood of $\eta_{\epsilon,j}'|_{[0,\sigma_{\epsilon,j}]}$ for some $j$.  Note that $\partial P_{\epsilon,j}$ consists of the right side of a flow line with angle $\theta_1$ and the left side of a flow line with angle $\theta_2$.  Consequently, an angle-varying flow line with angles contained in $[\theta_1,\theta_2]$ which changes angles only a finite number of times and at positive rational times cannot enter (the interior of) $P_{\epsilon,j}$ by the flow line interaction rules.  Thus if  $P_{\epsilon,j}$ for $1 \leq j \leq J_\epsilon$ is between the left and right boundaries of $\lightcone(\theta_1,\theta_2)$ then it is a subset of some element in $\CP_\epsilon(\theta_1,\theta_2)$.  This gives the first part of Part~\eqref{it::epsilon_pockets_contained}.  To establish the second part of Part~\eqref{it::epsilon_pockets_contained}, we first condition on $\lightcone(\theta_1,\theta_2)$.  Note that $\eta'$ enters the interior of a pocket of $\lightcone(\theta_1,\theta_2)$ at its opening point.  Thus, $\eta_{\epsilon,1}'(\tau_{\epsilon,1})$ must be on the boundary of such a pocket, say $P \in \CP_\epsilon(\theta_1,\theta_2)$.  Indeed, for otherwise the exploration used to generate $\eta_{\epsilon,1}'$ would have skipped following $\sideflow{1}{P}$.  Iterating this proves the claim for $j \geq 1$.

We turn to Part~\eqref{it::epsilon_pockets_lightcone}.  We fix $P \in \CP_\epsilon(\theta_1,\theta_2)$.  We claim that either $\eta_{\epsilon,1}^R = \sideflow{1}{P}$ or, if not, cannot merge into $\sideflow{1}{P}$.  To see that this is the case, we assume that $\eta_{\epsilon,1}^R$ is not equal to $\sideflow{1}{P}$.  If $\eta_{\epsilon,1}^R$ did merge into $\sideflow{1}{P}$, then $\eta'$ would visit the left side of $\sideflow{1}{P}$ before hitting $\open{P}$ because the path would have to visit the left side of $\eta_{\epsilon,1}^R$ before hitting $\open{P}$.  (This follows because whenever $\eta'$ hits the opening point of a pocket, the flow line interaction rules imply that it immediately enters and then exits at the closing point of the pocket.  Once it exits at the closing point, it immediately starts filling the $\theta_1$-angle boundary segment.)  This, in turn, would contradict the ordering because $\eta'$ would hit the left side of $\side{1}{P}$ before hitting $\open{P}$.  Iterating this argument implies that $\eta_{\epsilon,j}^R$ is either equal to $\sideflow{1}{P}$ where $P$ is the pocket of $\lightcone(\theta_1,\theta_2)$ which contains $P_{\epsilon,j}$ or does not merge with $\sideflow{1}{P}$.  Since the range of $\eta_\epsilon' = \eta_{\epsilon,J_\epsilon}'$ is equal to $\D \setminus \cup_{j=1}^{J_\epsilon} P_{\epsilon,j}$, if $\sideflow{1}{P}$ for some $P \in \CP_\epsilon(\theta_1,\theta_2)$ was not equal to one of the $\eta_{\epsilon,j}^R$ for $1 \leq j \leq J_\epsilon$, then $\eta_{\epsilon}'$ would have to visit $\open{P}$.  This is a contradiction since exploring $\sideflow{1}{P}$ upon hitting $\open{P}$ would lead to a pocket with diameter at least $\epsilon$ (since no other part of the $\theta_1$-angle boundary segment would have been explored by $\eta_{\epsilon}'$ before the path hits the opening point).  This proves that $\lightcone_\epsilon^R(\theta_1,\theta_2) \subseteq \CR_\epsilon$ almost surely.

We are now going to prove that the set which consists of those elements $\eta_{\epsilon,j}^R$ of $\CR_\epsilon$ for $1 \leq j \leq J_\epsilon$ which lie between the left and right boundaries of $\lightcone(\theta_1,\theta_2)$ is contained in $\lightcone_\epsilon^R(\theta_1,\theta_2)$ almost surely.  Fix $1 \leq j \leq J_\epsilon$.  Suppose that $P_{\epsilon,j}$ is strictly contained in the pocket $P$ of $\CP_\epsilon(\theta_1,\theta_2)$ which contains $P_{\epsilon,j}$.  Upon hitting the opening point $x_{\epsilon,j}$ of $P_{\epsilon,j}$, $\eta'$ has to enter into the interior of $P_{\epsilon,j}$ hence the interior of $P$ as explained above.  If $x_{\epsilon,j}$ is not equal to $\open{P}$, then this implies that $\eta'$ enters the interior of $P$ before hitting $\open{P}$.  This is a contradiction, therefore $P_{\epsilon,j} = P$ as desired.  This proves Part~\eqref{it::epsilon_pockets_lightcone}.
\end{proof}

\begin{proof}[Proof of Proposition~\ref{prop::ordering_local}]
As in the proof of Lemma~\ref{lem::epsilon_path_properties}, we let $\eta_\epsilon' = \eta_{\epsilon,J_\epsilon}'$.  By the construction and Part~\eqref{it::epsilon_pockets_lightcone} of Lemma~\ref{lem::epsilon_path_properties}, $\eta_\epsilon'$ visits the elements of $\CP_\epsilon(\theta_1,\theta_2)$ in the same order as $\eta$ defined just before Lemma~\ref{lem::continuous}.  Therefore it is easy to see from the construction that $\eta_\epsilon'$ with its excursions outside of the region between $\eta_1$ and $\eta_2$ converges uniformly modulo parameterization to $\eta$ as $\epsilon \to 0$.  Therefore Lemma~\ref{lem::epsilon_path_properties} implies that $\eta([0,t]) \cup \eta_1 \cup \eta_2$ is a local set for $h$ for each rational time $t$.  Combining this with the characterization of local sets given in the first part of \cite[Lemma~3.9]{SchrammShe10} implies that $\eta([0,\tau]) \cup \eta_1 \cup \eta_2$ is local for each $\eta$-stopping time $\tau$.
\end{proof}

We are now going to show that the law of the exploration path is continuous in the angles of the light cone.  This, in turn, will be used in Section~\ref{subsec::law} to establish the continuity of the law of $\SLE_\kappa(\rho)$ as the value of $\rho$ varies between $(-2-\tfrac{\kappa}{2}) \vee (\tfrac{\kappa}{2}-4)$ and $-2$.

\begin{proposition}
\label{prop::interpolation}
Suppose that $\theta_1 \leq \theta_2$ are angles with $\theta_2 - \theta_1 < \theta_c$ and $\theta_2 - \theta_1  \leq \pi$ and that $(\theta_n^1)$, $(\theta_n^2)$ are sequences of angles such that $\theta_n^1 \leq \theta_n^2$ and $\theta_n^2 - \theta_n^1 < \theta_c$ and $\theta_n^2 - \theta_n^1 \leq \pi$ for each $n \in \N$ and $\theta_n^j \to \theta_j$ as $n \to \infty$ for $j = 1,2$.  For each $n\in \N$, let $\eta_n$ be the path described above which visits the points of $\lightcone(\theta_n^1,\theta_n^2)$ and let $\eta$ be the path associated with $\lightcone(\theta_1,\theta_2)$.  Then $\eta_n \to \eta$ as $n \to \infty$ almost surely with respect to the uniform topology modulo reparameterization.
\end{proposition}

\begin{remark}
\label{rem::interpolation_not_continuous}
Proposition~\ref{prop::interpolation} does not imply that the map which takes a pair of angles $(\theta_1,\theta_2)$ to the exploration path of $\lightcone(\theta_1,\theta_2)$ is a continuous function into the space of paths equipped with the uniform topology modulo parameterization for a fixed realization of $h$.  This follows from the same reasoning as in Remark~\ref{rem::lightcone_not_continuous} in which it was explained that $(\theta_1,\theta_2) \mapsto \lightcone(\theta_1,\theta_2)$ is not a continuous function into the space of closed sets equipped with the Hausdorff topology for a fixed realization of $h$.  Proposition~\ref{prop::interpolation} does, however, imply that the map which takes a pair of angles $(\theta_1,\theta_2)$ to the law of the exploration path of $\lightcone(\theta_1,\theta_2)$ is continuous with respect to the weak topology.
\end{remark}

\begin{proof}[Proof of Proposition~\ref{prop::interpolation}]
For each $n \in \N$, let~$\eta_n'$ (resp.\ $\eta'$) be the counterflow line of~$h$ which orders~$\lightcone(\theta_1^n,\theta_2^n)$ (resp.\ $\lightcone(\theta_1,\theta_2)$) to generate the light cone exploration path~$\eta_n$ (resp.\ $\eta$).  Then we know that $\eta_n' \to \eta'$ almost surely as $n \to \infty$ with respect to the uniform topology.\footnote{This follows because if we fix any finite collection of points $z_1,\ldots,z_k \in \D$, the ``cells'' generated by the flow and dual flow lines corresponding to~$\eta_n'$ starting from these points will converge those of~$\eta'$ as $n \to \infty$.  If we fix enough points, then w.h.p.\ the maximal diameter of the cells will be smaller than a fixed choice of $\epsilon > 0$.  The claim follows by reparameterization~$\eta_n'$ so that it spends the same amount of time in a given cell is~$\eta'$ does.  Note that this time change converges to the identity as $n \to \infty$ since asymptotically the area of the cells converge, too.}  We also know from Proposition~\ref{prop::lightcone_continuous} that $\lightcone(\theta_n^1,\theta_n^2) \to \lightcone(\theta_1,\theta_2)$ almost surely as $n \to \infty$ with respect to the Hausdorff topology.  Fix an ordering $(r_j)$ of the points in~$\D$ with rational coordinates.  For each $n \in \N$, let~$(P_j^n)$ be the ordering of the pockets of~$\lightcone(\theta_1^n,\theta_2^n)$ according to diameter in which ties are broken according to which pocket contains the element of $(r_j)$ with the smallest index and let $(P_j)$ be the ordering of the pockets of~$\lightcone(\theta_1,\theta_2)$ defined in the same way.  For each $j,n \in \N$, we also let~$I_j^n$ (resp.\ $I_j$) be the interval of time in which~$\eta_n'$ (resp.\ $\eta'$) travels from the opening to the closing point of~$P_j^n$ (resp.\ $P_j$).  Note that~$I_j^n$ (resp.\ $I_j$) is also the interval of time in which~$\eta_n$ (resp.\ $\eta$) travels from the opening to the closing point of~$P_j^n$ (resp.\ $P_j$) along~$\side{1}(P_j^n)$ (resp.\ $\side{1}{P_j}$).  Let~$\eta_{1,j}^n = \eta_n|_{I_j^n}$ and~$\eta_{1,j} = \eta|_{I_j}$.  It follows from Proposition~\ref{prop::lightcone_continuous} and Proposition~\ref{prop::pocket_boundaries_continuous} that~$\eta_{1,j}^n \to \eta_{1,j}$ almost surely as~$n \to \infty$ with respect to the uniform topology modulo parameterization.  Combining all of the above, we can see that there exists~$k_0 \in \N$ such that for each~$k \geq k_0$ there exists~$n_0 \in \N$ such that the following is true.  We have that~$n \geq n_0$ implies that
\begin{enumerate}
\item the uniform distance modulo parameterization between~$\eta_{1,j}^n$ and~$\eta_{1,j}$ is at most~$\epsilon$ for each $1 \leq j \leq k$,
\item $\diam(P_j^n) \leq \epsilon$ for all $j > k$, and
\item $\|\eta_n' - \eta'\|_\infty \leq \epsilon$.
\end{enumerate}
Reparameterizing the time of $\eta_n'$ and $\eta_n$ so that $I_j^n = I_j$ for each $1 \leq j \leq k$, it thus follows that, after possibly reparameterizing the time of $\eta_n'$ and $\eta_n$ within each $I_j$, with $\CI = \cup_{1 \leq j \leq k} I_j$ we have that
\begin{equation}
\label{eqn::big_pocket_distance}
  \|\eta_n|_\CI - \eta|_\CI  \|_\infty \leq \epsilon.
\end{equation}
Let $\CJ = (\eta_n')^{-1}(\cup_{j > k} P_j^n) = (\eta')^{-1}(\cup_{j > k} P_j)$.  By the way that we have defined the light cone exploration path, we also have that
\begin{equation}
\label{eqn::small_pocket_distance}
 \|\eta_n'|_{\CJ} - \eta_n|_{\CJ}\|_\infty \leq \epsilon \quad\text{and}\quad \|\eta'|_{\CJ} - \eta|_{\CJ}\| \leq \epsilon.
 \end{equation}
Note that $\eta_n$ (resp.\ $\eta$) is determined by its values on $\CI \cup \CJ$ since the times in $[0,\infty) \setminus \ol{\CI \cup \CJ}$ correspond to those times in which $\eta_n'$ (resp.\ $\eta'$) makes an excursion from $\side{1}{P_j^n}$ (resp.\ $\side{1}{P_j}$) into $P_j^n$ (resp.\ $P_j$) for some $1 \leq j \leq k$.  In particular, $\eta_n$ (resp.\ $\eta$) is piecewise constant in $[0,\infty) \setminus \ol{\CI \cup \CJ}$.  Combining \eqref{eqn::big_pocket_distance} and \eqref{eqn::small_pocket_distance} implies that
\begin{align*}
          \| \eta_n - \eta\|_\infty
&\leq \| \eta_n|_{\CI} - \eta|_{\CI} \|_\infty + \| \eta_n|_{\CJ} - \eta|_{\CJ} \|_\infty\\
&\leq \epsilon + \| \eta_n|_{\CJ} - \eta_n'|_{\CJ} \|_\infty + \|\eta_n'|_{\CJ} - \eta'|_{\CJ} \|_\infty + \| \eta'|_{\CJ} - \eta|_{\CJ} \|_\infty\\
&\leq 3\epsilon + \| \eta_n' - \eta'\|_\infty \leq 4\epsilon,
\end{align*} 
which gives the desired result.
\end{proof}

\subsection{Law of the exploration path}
\label{subsec::law}

\begin{figure}[ht!]
\begin{center}
\includegraphics[scale=0.85]{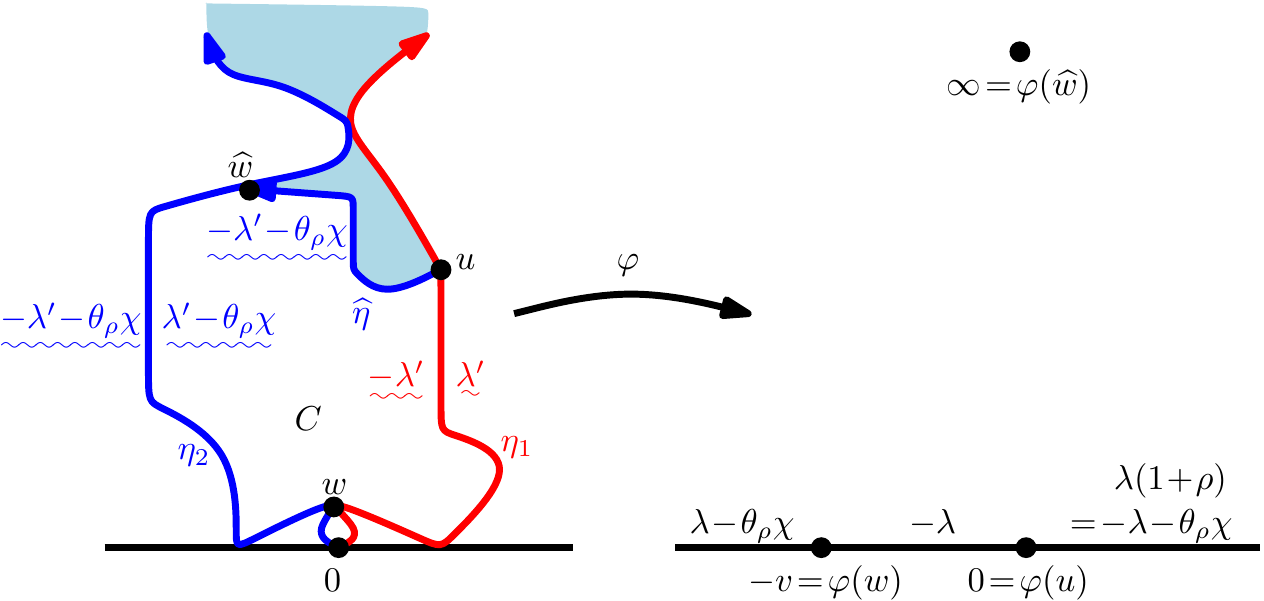}
\end{center}
\caption{\label{fig::continuum_lightcone}
Suppose that~$h$ is a GFF on~$\h$ with piecewise constant boundary data which changes values at most a finite number of times.  Shown on the left side are the flow lines $\eta_1,\eta_2$ with angles $0,\theta_\rho$, respectively, of~$h$ starting from~$0$, both of which we assume reach $\infty$ before hitting the continuation threshold, and the flow line~$\wh{\eta}$ of angle~$\theta_\rho$ starting from a point $u$ on the boundary of a component of $\h \setminus (\eta_1 \cup \eta_2)$ which is between $\eta_1$ and~$\eta_2$.  The outer boundary of $\lightcone(0,\theta_\rho)$ is given by $\eta_1 \cup \eta_2$.  The exploration path $\eta$ of $\lightcone(0,\theta_\rho)$ starts from $\infty$ and its outer boundary stopped upon hitting $u$ is equal to the union of $\wh{\eta}$ and the part of $\eta_1$ (resp.\ $\eta_2$) after it hits $u$ (resp.\ $w$).  The light blue region indicates the hull of $\eta$ stopped upon hitting $u$.  Let $C$ be the component surrounded by $\eta_1$, $\eta_2$, and $\wh{\eta}$ as shown and let $\varphi \colon C \to \h$ be the conformal map which takes $u$ to $0$, $\wh{w}$, the point where $\eta_2$ and $\wh{\eta}$ merge, to $\infty$, and $w$, the point on $\partial C$ where $\eta_1,\eta_2$ first intersect, to $-v$.  The boundary data for $\wt{h} = h \circ \varphi^{-1} - \chi \arg(\varphi^{-1})'$ is as shown on the right.  The image of the part of $\lightcone(0,\theta_\rho)$ contained in $\ol{C}$ under $\varphi$ is equal to the light cone $\lightcone_{\R_-}(0,\theta_\rho)$ of $\wt{h}$ starting from $0$ and the image of the part of $\eta$ when it is in $\ol{C}$ gives the corresponding exploration path.  Sending $v \to \infty$, $\lightcone_{\R_-}(0,\theta_\rho)$ converges to the corresponding light cone of a field whose boundary conditions are given by $-\lambda$ (resp.\ $\lambda(1+\rho)$) on $\R_-$ (resp.\ $\R_+$).
}
\end{figure}

\begin{figure}[ht!]
\begin{center}
\includegraphics[scale=0.85]{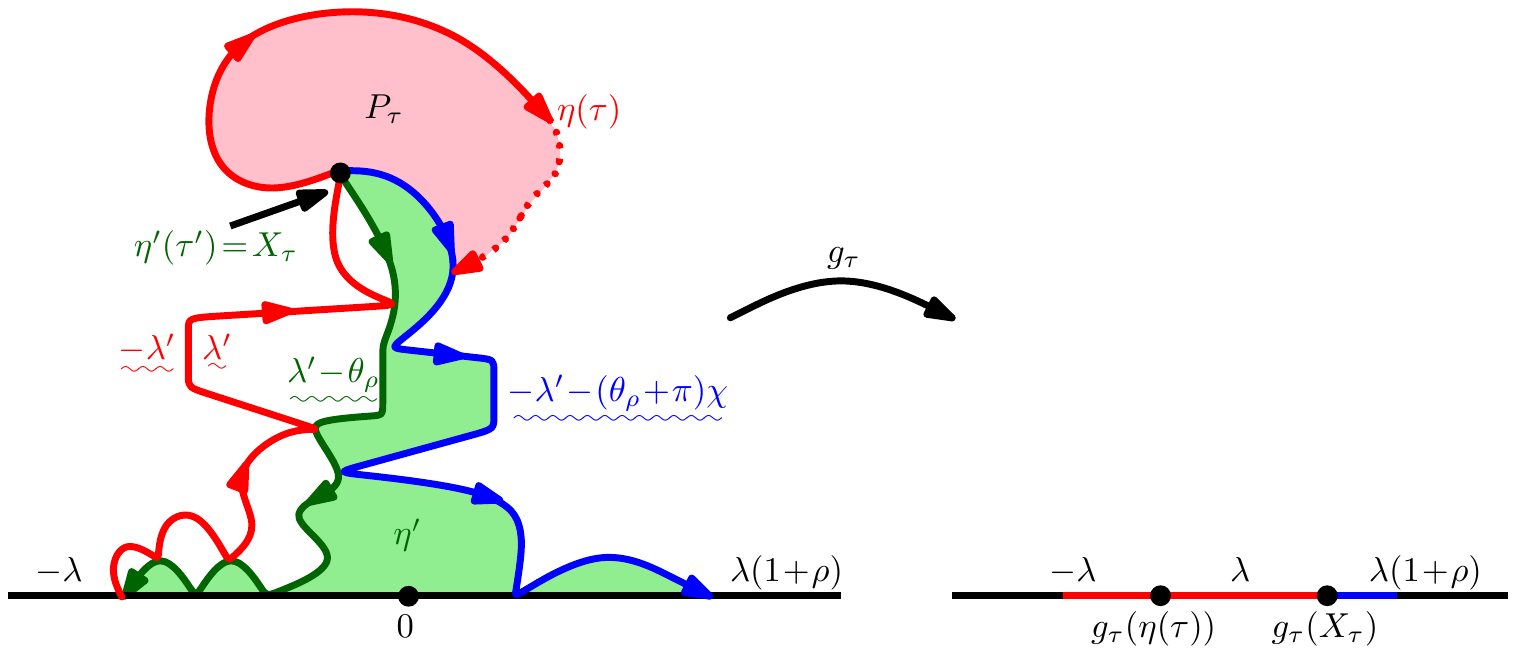}
\end{center}
\caption{\label{fig::outer_boundary} Setup for the proof of Lemma~\ref{lem::lightcone_conditional}.  Suppose that $h$ is a GFF on $\h$ with the illustrated boundary data where $\rho \in [\tfrac{\kappa}{2}-4,-2)$ and that $\eta$ is the exploration path associated with $\lightcone_{\R_-}(0,\theta_\rho)$ where $\theta_\rho = \pi(\rho+2)/(\tfrac{\kappa}{2}-2)$.  Suppose that $\tau$ is a stopping time for $\eta$.  Then we can describe the boundary behavior of the conditional law of $h$ given $\eta|_{[0,\tau]}$ restricted to the unbounded connected component of $\h \setminus \eta([0,\tau])$ by relating the outer boundary of $\eta([0,\tau])$ (the union of the red and blue paths in the illustration) to the outer boundary of the counterflow line $\eta'$ (the hull of which is indicated in light green) stopped at the first time $\tau'$ that it hits $X_\tau$, the opening point of the pocket whose boundary is being drawn by $\eta$ at time $\tau$, and the $0$-angle flow line starting from the leftmost point of $\eta'([0,\tau']) \cap \R$ (red).  The region bounded by the solid red, dashed red, and blue paths is the pocket of $\lightcone_{\R_-}(0,\theta_\rho)$ whose $0$-angle boundary is being drawn by $\eta$ at time $\tau$.}
\end{figure}

It will be more convenient for us to work on $\h$ in this section.  Throughout, we fix $\rho \in [\tfrac{\kappa}{2}-4,-2)$ and suppose that $h$ is a GFF on $\h$ with boundary conditions given by $-\lambda$ on $\R_-$ and $\lambda(1+\rho)$ on $\R_+$, as shown in Figure~\ref{fig::outer_boundary}.  Let $\theta_\rho$ be as in \eqref{eqn::lightcone_angle}.  Let $\eta'$ be the counterflow line starting from the origin whose left boundary stopped upon hitting a point $z$ is equal to the flow line with angle $\theta_\rho$ starting from $z$.  Explicitly, $\eta'$ is the counterflow line of $h+(\tfrac{\pi}{2}+\theta_\rho)\chi$ starting from the origin.  Note that this is the ``same'' as the corresponding counterflow line starting from $\infty$ because the path starting from $\infty$ will trace along $\R_+$ and does not enter (the interior of) $\h$ until hitting the origin.  Using exactly the same analysis as in Section~\ref{subsec::lightcones} and Section~\ref{subsec::explorations}, we can construct from $\eta'$ a path $\eta$ which explores $\lightcone_{\R_-}(0,\theta_\rho)$.  This path is continuous, has a continuous chordal Loewner driving function, and is almost surely determined by $h$.  Moreover, the path drawn up to any stopping time is local for $h$ (in contrast to Proposition~\ref{prop::ordering_local}, it is not necessary also to condition on the outer boundary of the light cone).  That these properties hold follows from the results of the previous subsections and the conditioning argument explained in Figure~\ref{fig::continuum_lightcone}.

We will now determine the law of $\eta$.  This, in turn, will lead to the proofs of Theorem~\ref{thm::continuous} and Theorem~\ref{thm::coupling} (it does not quite imply Theorem~\ref{thm::interpolation} because the boundary data is different for different $\rho$ values).  For each $t \geq 0$, let $K_t$ be the closure of the complement of the unbounded connected component $\h_t$ of $\h \setminus \eta([0,t])$.  For each $t \geq 0$ such that $\eta$ is drawing a segment of $\side{1}{P}$ where $P$ is a pocket of $\lightcone_{\R_-}(0,\theta_\rho)$ in an open interval of time containing $t$, let $P_t$ be the corresponding pocket and let $X_t$ be its opening point.  For other values of $t$, we take $P_t = \emptyset$ and let $X_t$ be the limit as $s \downarrow t$ where the times $s$ are restricted to those in which $\eta$ is drawing a segment of $\side{1}{P}$ for a pocket $P$ of $\lightcone_{\R_-}(0,\theta_\rho)$.  The main step in determining the law of $\eta$ is the following, which gives the conditional law of $h$ given $\eta$ drawn up to a fixed stopping time.

\begin{lemma}
\label{lem::lightcone_conditional}
Suppose that $\tau$ is an almost surely finite stopping time for $\eta$.  Then the conditional law of $h$ given $\eta|_{[0,\tau]}$ is independently that of a GFF in each of the components of $\h \setminus \eta([0,\tau])$.  The boundary conditions in each of the bounded components agrees with that of $h$ given $\lightcone_{\R_-}(0,\theta_\rho)$ in the corresponding component (recall Lemma~\ref{lem::form_pockets}).  On $\partial \h_\tau$, the boundary conditions are given by:
\begin{enumerate}[(i)]
\item\label{it::bd1} the left side of a $0$-angle flow line on the segment of $\partial \h_\tau$ which is to the left of $\eta(\tau)$ (left side of the red path in Figure~\ref{fig::outer_boundary}),
\item\label{it::bd2} the right side of a $0$-angle flow line on the right side of the segment of $\partial \h_\tau$ from $\eta(\tau)$ to $X_\tau$ (counterclockwise direction; right side of red path in Figure~\ref{fig::outer_boundary}), and
\item\label{it::bd3} the left side of a $\theta_\rho$-angle flow line on the segment from $X_\tau$ to $\R_+$ (counterclockwise direction; left side of blue path in Figure~\ref{fig::outer_boundary}).
\end{enumerate}
\end{lemma}
\begin{proof}
Let $\tau$ be any almost surely finite stopping time for $\eta$ such that $\eta(\tau)$ is contained in the interior of a $0$-angle boundary segment of a pocket of $\lightcone_{\R_-}(0,\theta_\rho)$.  It suffices to show that the conditional law of $h$ given $\eta|_{[0,\tau]}$ is as described in the statement of the proposition for stopping times $\tau$ of this form.  Indeed, we know that stopping times of this form are dense in $[0,\infty)$ by the proof of Lemma~\ref{lem::continuous} and, by Proposition~\ref{prop::ordering_local}, we know that $\eta([0,\sigma])$ is a local set for $h$ for every $\eta$-stopping time $\sigma$, so we can use the continuity result for local sets proved in \cite[Proposition~6.5]{MS_IMAG}.  The statement regarding the conditional law of $h$ restricted to the components which are surrounded by $\eta([0,\tau])$ follows from \cite[Proposition~3.8]{MS_IMAG} by comparing to $\lightcone_{\R_-}(0,\theta_\rho)$.

We are now going to describe the boundary behavior for $h$ on $\partial \h_\tau$ using \cite[Proposition~3.8]{MS_IMAG} and a construction involving $\eta'$ and some auxiliary paths.  See Figure~\ref{fig::outer_boundary} for an illustration of the setup of the proof.  Let $\tau'$ be the first time that $\eta'$ hits $X_\tau$.  It follows from the way that we constructed the ordering of $\lightcone_{\R_-}(0,\theta_\rho)$ that the left boundary of $\eta'([0,\tau'])$ is contained in $\eta([0,\tau])$ and is in fact equal to the segment of $\partial \h_\tau$ which connects $X_\tau$ to $\R_+$ in the counterclockwise direction (left side of blue path in Figure~\ref{fig::outer_boundary}).  Suppose that $t \in \Q_+$.  On the event $\{t < \tau'\}$, we can use \cite[Proposition~3.8]{MS_IMAG} to get that the boundary behavior of $h$ given $\eta|_{[0,\tau]}$ on the segment of $\partial \h_\tau$ which is to the right of $X_\tau$ and contained in $\eta'([0,t])$ is as claimed in \eqref{it::bd3}.  This proves the boundary behavior claimed in \eqref{it::bd3} because by continuity and because this holds for all $t \in \Q_+$ simultaneously almost surely.

For each $s \in \Q_+$, we let $A_s = \eta'([0,s]) \cup \eta_s$ where $\eta_s$ is the $0$-angle flow line of the conditional GFF $h$ given $\eta'|_{[0,s]}$ starting from the leftmost point of $\eta'([0,s]) \cap \R$.  Note that $\eta_s$ reflects off the right boundary of $\eta'([0,s])$.  We are now going to establish the boundary behavior claimed in \eqref{it::bd1} by showing that there almost surely exists $s \in \Q_+$ such that the segment of $\partial \h_\tau$ which is to the left of $\eta(\tau)$ is contained in $\eta_s$.  This will also give \eqref{it::bd2}.  Indeed, this suffices since we can use \cite[Proposition~3.8]{MS_IMAG} to compare the boundary behavior of $h$ given $\eta|_{[0,\tau]}$ to that of $h$ given $A_s$.

We are now going to show that $\eta_s$ is equal to the closure $C_s$ of the $0$-angle boundaries of the pockets of $\lightcone_{\R_-}(0,\theta_\rho)$ which intersect the right boundary $R_s'$ of $\eta'([0,s])$ (dark green path in Figure~\ref{fig::outer_boundary}).  We will first show that $\eta_s$ is (non-strictly) to the left of~$C_s$.  Fix a countable, dense set~$D$ in~$R_s'$.  If $z \in D$ then \cite[Theorem~1.5]{MS_IMAG} implies that~$\eta_s$ is to the left of the $0$-angle flow line of $h$ given $\eta'|_{[0,s]}$ starting from~$z$.  Since~$D$ is countable, this holds for all $z \in D$ simultaneously almost surely.  Moreover, it is easy to see that $\side{1}{P}$ for a pocket $P$ of $\lightcone_{\R_-}(0,\theta_\rho)$ which intersects $R_s'$ can be written as a limit of $0$-angle flow lines starting from points in $D$ by taking starting points contained $P \cap R_s'$ which get progressively closer to $\open{P}$.  Indeed, this follows since such a flow line will merge with $\side{1}{P}$ upon intersecting it by \cite[Theorem~1.5]{MS_IMAG}.  This proves that $\eta_s$ is (non-strictly) to the left of $C_s$.  We will next argue that $\eta_s$ is (non-strictly) to the right of (and hence equal to) $C_s$.  Indeed, the reason for this is that the flow line interaction rules imply that an angle-varying flow line with angles contained in $[0,\theta_\rho]$ cannot enter into a pocket formed by $\eta_s$ and $\eta'([0,s])$.  This proves the assertion and hence the claim that $\eta_s = C_s$.

Take $s \in \Q_+$ with $s > \tau$ such that $\eta'([0,s])$ has not hit the closing point of the pocket of $\lightcone_{\R_-}(0,\theta_\rho)$ whose opening point is given by $X_\tau$.  Note that $\eta$ visits a pocket $P$ of $\lightcone_{\R_-}(0,\theta_\rho)$ before time $\tau$ if and only if $\eta'$ visits the interior of $P$ before time $\tau'$.  Consequently, it is easy to see that the boundary segments referred to in \eqref{it::bd1} and \eqref{it::bd2} are contained in $\eta_s$.  This proves the desired result by invoking \cite[Proposition~3.8]{MS_IMAG}.
\end{proof}

Now that we have determined by the boundary behavior for the conditional law of $h$ given $\eta|_{[0,\tau]}$ up to any stopping time $\tau$, we can now give the law of $\eta$.

\begin{lemma}
\label{lem::law_of_path}
The law of $\eta$ is given by that of an $\SLE_\kappa(\rho)$ process in $\h$ from $0$ to $\infty$ where
\begin{equation}
\label{eqn::rho_theta_relation}
 \rho = \ol{\theta}_\rho \left( \frac{\kappa}{2}-2\right) - 2 \quad\text{and}\quad \ol{\theta}_\rho = \frac{\theta_\rho}{\pi}.
\end{equation}
\end{lemma}
\begin{proof}
The martingale characterization of the $\SLE_\kappa(\rho)$ processes given in \cite[Theorem~2.4]{MS_IMAG} combined with Lemma~\ref{lem::lightcone_conditional} implies that~$\eta$ evolves as an $\SLE_\kappa(\rho)$ process with the value of~$\rho$ determined by~$\theta_\rho$ as given in \eqref{eqn::rho_theta_relation} in those time intervals in which~$\eta$ is not intersecting the past of its range, i.e., those times~$t$ such that $\eta(t) \notin \eta([0,t))$.  For each $t$, let $Z_t = g_t(X_t)$.  This implies that $Z - W$ evolves as $\sqrt{\kappa}$ times a Bessel process of dimension $d(\kappa,\rho) = 1+\tfrac{2(\rho+2)}{\kappa}$ during these times.  By Lemma~\ref{lem::continuous_loewner}, we know that $\eta$ has a continuous Loewner driving function, from which it follows that $Z_t-W_t$ is instantaneously reflecting at $0$.  Therefore $Z - W$ evolves as $\sqrt{\kappa}$ times a Bessel process of dimension $d(\kappa,\rho)$ for all $t \geq 0$.  The result then follows by applying Proposition~\ref{prop::bessel_pv}.
\end{proof}

\begin{proof}[Proof of Theorem~\ref{thm::continuous} and Theorem~\ref{thm::coupling}]
By Lemma~\ref{lem::law_of_path}, we know that $\eta$ is an $\SLE_\kappa(\rho)$ process with the desired value of $\rho$ and by Lemma~\ref{lem::lightcone_conditional} we know that $\eta$ is coupled with and almost surely determined by the field as described in Theorem~\ref{thm::coupling}.
\end{proof}

Now that we have proved Theorem~\ref{thm::continuous} and Theorem~\ref{thm::coupling}, it is left to prove Theorem~\ref{thm::interpolation}.  The result does not immediately follow from Proposition~\ref{prop::interpolation} because that result describes what happens to the light cone path when we change the angles of the light cone but leave the GFF is fixed.  In the present setting, we are changing the angles of the light cone \emph{and} the boundary data of the GFF.

\begin{figure}
\begin{center}
\includegraphics[scale=0.85]{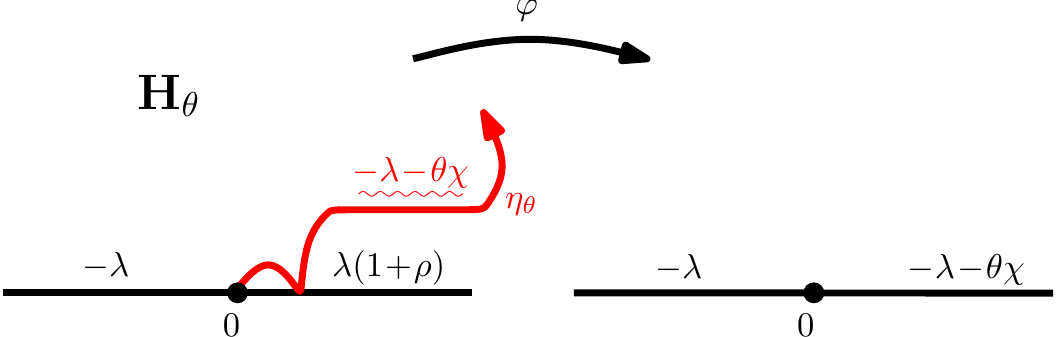}
\end{center}
\caption{\label{fig::continuous_interpolation}  Illustration of the idea of the proof of Theorem~\ref{thm::interpolation}.  Suppose that $h$ is a GFF on $\h$ with the illustrated boundary data.  Then $h$ is compatible with a coupling with an $\SLE_\kappa(\rho)$ process $\eta$ starting from $0$.  Fix $\theta > 0$ and let $\eta_\theta$ be the flow line of $h$ starting from $0$ with angle $\theta$ and let $\h_\theta$ be the component of $\h \setminus \eta_\theta$ which is to the left of $\eta_\theta$.  With $\varphi \colon \h_\theta \to \h$ a conformal transformation which fixes $0$ and $\infty$, $h \circ \varphi^{-1} - \chi \arg(\varphi^{-1})'$ is a GFF on $\h$ with the boundary data shown on the right.  Since the law of~$\eta_\theta$ is continuous in~$\theta$ and the light cone exploration path is continuous in its angles, we get the desired interpolation result for $\SLE_\kappa(\rho)$.}
\end{figure}

\begin{proof}[Proof of Theorem~\ref{thm::interpolation}]
We are going to extract the result in two steps by first applying Proposition~\ref{prop::interpolation} and then using a conditioning argument.  (This is similar in spirit to our proof of the continuity of the $\SLE_\kappa(\rho)$ processes for $\rho > -2$ given in \cite{MS_IMAG}.)  Let $\Psi  \colon \h \to \D$ be a conformal transformation with $\Psi(0) = -i$ and $\Psi(\infty) = i$.  Fix $\rho \in [\tfrac{\kappa}{2}-4,-2)$ with $\rho > -2-\tfrac{\kappa}{2}$ and suppose that $h$ is a GFF on $\h$ with boundary conditions which are given by $-\lambda$ on $\R_-$ and $\lambda(1+\rho)$ on $\R_+$.  Then~$h$ is a compatible with a coupling with an $\SLE_\kappa(\rho)$ process~$\eta$ from~$0$ to~$\infty$ as in Theorem~\ref{thm::coupling}.  Moreover, $\eta$ is equal to the light cone exploration path associated with $\lightcone_{\R_-}(0,\theta_\rho)$ where $\theta_\rho =  \tfrac{\pi(\rho+2)}{\kappa/2-2}$.  For each $\theta \geq \theta_\rho$, let~$\eta^\theta$ be the light cone path associated with $\lightcone_{\R_-}(0,\theta)$.  By Proposition~\ref{prop::interpolation}, we know that $\Psi(\eta^\theta) \to \Psi(\eta)$ uniformly (modulo reparameterization) as $\theta \downarrow \theta_\rho$.  For each $\theta \geq \theta_\rho$, we let~$\eta_\theta$ be the flow line of~$h$ with angle~$\theta$ starting from~$0$.  (For $\theta=\theta_\rho$, we take $\eta_\theta$ to be equal to $\R_+$.)  Then we know that $\eta_\theta \to \eta_{\theta_\rho}$ locally uniformly as $\theta \downarrow \theta_\rho$ almost surely.  Let $\varphi_\theta$ be the conformal transformation which takes the component $\h_\theta$ of $\h \setminus \eta_\theta$ which is to the left of $\eta_\theta$ to $\h$ fixing $0$, $-1$, and $\infty$.  Then $\Psi \circ \varphi_\theta^{-1} \circ \Psi^{-1}$ converges locally uniformly to the identity on $\D$ almost surely as $\theta \downarrow \theta_\rho$.  Note that the boundary conditions for the GFF $h_\theta = h \circ \varphi_\theta^{-1} - \chi \arg (\varphi_{\theta}^{-1})'$ are given by $-\lambda$ on $\R_-$ and by $-\lambda-\theta \chi$ on $\R_+$.  Since $\varphi_\theta(\eta^\theta)$ is the light cone path associated with the light cone with angle range $[0,\theta]$ of $h_\theta$, we know that $\varphi_\theta(\eta^\theta)$ is an $\SLE_\kappa(\rho_\theta)$ process where $\rho_\theta = \tfrac{\theta}{\pi}(\tfrac{\kappa}{2}-2)-2$.  The desired result follows since combining everything implies that $\Psi(\varphi_\theta(\eta^\theta)) \to \Psi(\eta)$ almost surely as $\theta \downarrow \theta_\rho$.  The continuity when $\theta \uparrow \theta_\rho$ is proved similarly.
\end{proof}

\section{Behavior at the boundary of the light cone regime}
\label{sec::limiting_cases}

\begin{figure}[ht!]
\begin{center}
\includegraphics[scale=0.85]{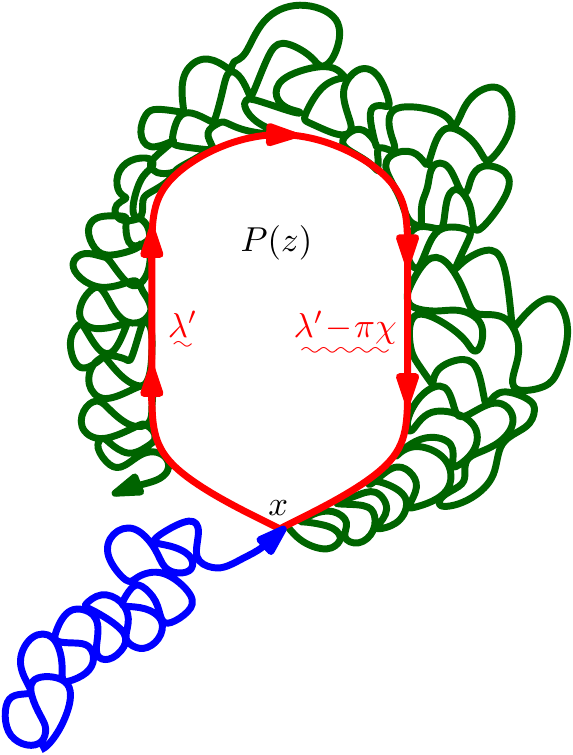}
\vspace{-0.02\textheight}
\end{center}
\caption{\label{fig::pocket_order} Illustration of the order in which an $\SLE_\kappa(\tfrac{\kappa}{2}-4)$ process $\eta$ visits the points in its range.  Shown is a pocket $\pocket{z}$ of $\eta$ with opening point $x$ and a clockwise orientation.  Note that $\partial \pocket{z}$ is given by a $0$-angle flow line loop starting from $x$.  The blue path indicates $\eta$ up until hitting $x$.  Upon hitting $x$, $\eta$ immediately traces $\partial \pocket{z}$ in the clockwise direction.  The green path indicates the range of $\eta$ after it finishes drawing $\partial \pocket{z}$.  This part of the path will crawl along $\partial \pocket{z}$ in the counterclockwise direction.  In contrast, the $\SLE_{\kappa'}(\tfrac{\kappa'}{2}-4)$ counterflow line $\eta'$ whose range is equal to $\eta$ (see Proposition~\ref{prop::lightcone_counterflow}) will draw $\partial \pocket{z}$ in the \emph{opposite} (counterclockwise) direction and, while doing so, visits the pockets in its range which intersect $\partial \pocket{z}$.}
\end{figure}

\begin{figure}[ht!]
\begin{center}
\includegraphics[scale=0.85]{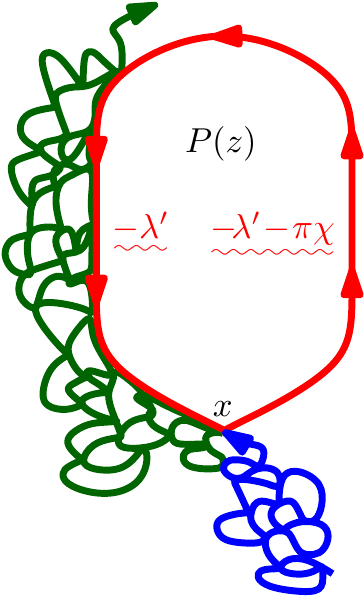}
\vspace{-0.02\textheight}
\end{center}
\caption{\label{fig::pocket_order2} (Continuation of Figure~\ref{fig::pocket_order}.) Shown is a pocket $\pocket{z}$ of $\eta$ with opening point which has a counterclockwise orientation.  Note that $\partial \pocket{z}$ is given by a $\pi$-angle flow line loop starting from $x$.  The blue path segment indicates the part of $\eta$ up until it hits $x$ and the green path segment indicates part of $\eta$ as it draws $\partial \pocket{z}$.  In contrast to the case of a clockwise loop, as considered Figure~\ref{fig::pocket_order}, $\eta$ visits the points on $\partial \pocket{z}$ in the same order as $\eta'$.  Moreover, as it does so, it draws the boundaries of the pockets which intersect $\partial \pocket{z}$.}
\end{figure}

We are now going to describe the behavior of $\SLE_\kappa(\rho)$ at the threshold $\rho = \tfrac{\kappa}{2}-4$ which lies between the light cone and trunk regimes.  When $\rho = \tfrac{\kappa}{2}-4$, the opening angle for the light cone is equal to $\pi$.  Note that $\pi < \theta_c$ if and only if $\kappa \in (2,4)$.  As we mentioned earlier, this is closely connected with the fact that an $\SLE_{\kappa'}$ process is space-filling if and only if $\kappa' \geq 8$.  In analogy with \cite[Theorem~1.4]{MS_IMAG}, in this case, the range of the path is equal to that of a form of an $\SLE_{\kappa'}$ process as stated in the following proposition.

\begin{proposition}
\label{prop::lightcone_counterflow}
Suppose that $\kappa \in (2,4)$ (so that $\pi < \theta_c$) and let $\eta$ be an $\SLE_\kappa(\tfrac{\kappa}{2}-4)$ process in $\h$ from $0$ to $\infty$ with a single force point located at $0^+$.  Then the range of $\eta$ is equal in law to that of an $\SLE_{\kappa'}(\tfrac{\kappa'}{2}-4)$ process $\eta'$ in $\h$ from $0$ to $\infty$ where the force point is located at $0^-$.
\end{proposition}
\begin{remark}
\label{rem::hits_points_differently}
We emphasize that the statement of Proposition~\ref{prop::lightcone_counterflow} is that the law of the \emph{range} of~$\eta$ is equal to the law of the \emph{range} of~$\eta'$.  As explained in Figure~\ref{fig::pocket_order} and Figure~\ref{fig::pocket_order2}, the order in which the paths visit the points in their common range is different.
\end{remark}
\begin{proof}[Proof of Proposition~\ref{prop::lightcone_counterflow}]
Suppose that~$h$ is a GFF on $\h$ with boundary data given by $-\lambda$ (resp.\ $-\lambda-\pi\chi$) on~$\R_-$ (resp.\ $\R_+$) and let $\eta$ be the $\SLE_\kappa(\tfrac{\kappa}{2}-4)$ process coupled with $h$ as the light cone path from $0$ to $\infty$ as in Theorem~\ref{thm::coupling}.  Note that
\[ -\lambda = -\lambda'-\frac{\pi \chi}{2} \quad\text{and}\quad -\lambda-\pi\chi = -\lambda' - \frac{3\pi \chi}{2}.\]
Let $\eta'$ be the counterflow line of $h+3\pi \chi/2$ starting from~$0$.  Then~$\eta'$ is an $\SLE_{\kappa'}(\tfrac{\kappa'}{2}-4)$ process where the force point is located at~$0^-$.  By \cite[Theorem~1.13]{MS_IMAG4}, we note that the left boundary of~$\eta'$ stopped upon hitting a point $z \in \h$ is equal to the flow line of~$h$ starting from $z$ with angle~$\pi$.  Consequently, it follows that the range of~$\eta'$ is equal to the range of~$\eta$.
\end{proof}

We finish by recording two immediate consequences of Proposition~\ref{prop::lightcone_counterflow}.

\begin{corollary}
\label{cor::boundary_filling}
Suppose that $\kappa \in (2,4)$ and let~$\eta$ be an $\SLE_\kappa(\tfrac{\kappa}{2}-4)$ process in~$\h$ from~$0$ to~$\infty$ with a single boundary force point located at~$0^+$.  Then~$\R_-$ is almost surely contained in the range of~$\eta$.
\end{corollary}
\begin{proof}
This follows from Proposition~\ref{prop::lightcone_counterflow} and the fact that $\tfrac{\kappa'}{2}-4$ is the critical value of~$\rho$ at or below which a counterflow line is boundary filling.  In particular, with~$\eta'$ as in the statement of Proposition~\ref{prop::lightcone_counterflow}, we have that $\R_-$ is contained in the range of~$\eta'$.
\end{proof}

\begin{corollary}
\label{cor::critical_pocket_structure}
Suppose that $\kappa \in (2,4)$ and let $\eta$ be an $\SLE_\kappa(\tfrac{\kappa}{2}-4)$ process in~$\h$ from~$0$ to $\infty$ with a single boundary force point located at $0^+$ coupled with a GFF~$h$ on~$\h$ with boundary data equal to $-\lambda$ (resp.\ $-\lambda-\pi \chi$) on $\R_-$ (resp.\ $\R_+$).  If~$\eta$ separates~$z$ from $\partial \h$, then $\partial \pocket{z}$ is equal to the flow line of $h$ with angle~$0$ (resp.\ $\pi$) starting from $\open{z}$ if~$\eta$ traverses $\partial \pocket{z}$ with a clockwise (resp.\ counterclockwise) orientation.  In particular, the boundaries of the pockets of~$\eta$ have only one side.
\end{corollary}
\begin{proof}
This follows from Proposition~\ref{prop::lightcone_counterflow} since the same is true for the counterflow line~$\eta'$ (see, e.g.\ \cite[Theorem~1.13]{MS_IMAG4}).
\end{proof}

\bibliographystyle{hmralphaabbrv}
\bibliography{sle_kappa_rho}

\bigskip

\filbreak
\begingroup
\small
\parindent=0pt

\bigskip
\vtop{
\hsize=5.3in
Statistical Laboratory, DPMMS\\
University of Cambridge\\
Cambridge, UK}

\bigskip
\vtop{
\hsize=5.3in
Department of Mathematics\\
Massachusetts Institute of Technology\\
Cambridge, MA, USA } \endgroup \filbreak

\end{document}